\RequirePackage{fix-cm}

\documentclass[smallcondensed]{svjour3}     

\smartqed  

\usepackage{graphicx}

\usepackage{amsmath}
\usepackage{amssymb}
\usepackage{hyperref}
\usepackage{booktabs}
\usepackage{mathtools} \mathtoolsset{showonlyrefs=true} \MakeRobust{\eqref}
\usepackage{textcomp}
\usepackage{enumerate}
\usepackage{bm}
\usepackage[mathscr]{eucal}
\usepackage{lmodern}
\usepackage{scalerel}
\usepackage{xr}
\usepackage{algorithm}
\usepackage{algorithmic} 
\usepackage[round]{natbib}
\usepackage{paralist}

\usepackage[usenames,dvipsnames,svgnames]{xcolor}

\hypersetup{pdfauthor={},
	pdftitle={Fast Randomized Matrix and Tensor Interpolative Decomposition Using CountSketch},
	colorlinks=true,
	citecolor=MidnightBlue,
	urlcolor=Bittersweet,
	linkcolor=MidnightBlue,
}

\spnewtheorem{fact}{Fact}{\bf}{\rm}

\newcommand{\F}{\textup{F}}

\newcommand\Cb{\mathbb{C}}
\newcommand\Eb{\mathbb{E}}

\newcommand\Pb{\mathbb{P}}

\newcommand\Rb{\mathbb{R}}

\newcommand{\zerobf}[0]{\mathbf{0}}
\newcommand{\Abf}[0]{\mathbf{A}}
\newcommand{\Bbf}[0]{\mathbf{B}}
\newcommand{\Cbf}[0]{\mathbf{C}}
\newcommand{\Dbf}[0]{\mathbf{D}}
\newcommand{\Ebf}[0]{\mathbf{E}}
\newcommand{\Fbf}[0]{\mathbf{F}}
\newcommand{\Gbf}[0]{\mathbf{G}}

\newcommand{\Ibf}[0]{\mathbf{I}}
\newcommand{\Lbf}[0]{\mathbf{L}}
\newcommand{\Mbf}[0]{\mathbf{M}}
\newcommand{\Pbf}[0]{\mathbf{P}}
\newcommand{\Qbf}[0]{\mathbf{Q}}
\newcommand{\Rbf}[0]{\mathbf{R}}
\newcommand{\Sbf}[0]{\mathbf{S}}
\newcommand{\Tbf}[0]{\mathbf{T}}
\newcommand{\Ubf}[0]{\mathbf{U}}
\newcommand{\Vbf}[0]{\mathbf{V}}
\newcommand{\Xbf}[0]{\mathbf{X}}
\newcommand{\Ybf}[0]{\mathbf{Y}}
\newcommand{\Zbf}[0]{\mathbf{Z}}
\newcommand{\abf}[0]{\mathbf{a}}

\newcommand{\ibf}[0]{\mathbf{i}}
\newcommand{\jbf}[0]{\mathbf{j}}

\newcommand{\ubf}[0]{\mathbf{u}}
\newcommand{\vbf}[0]{\mathbf{v}}
\newcommand{\xbf}[0]{\mathbf{x}}

\newcommand{\Sigmabf}[0]{\bm{\Sigma}}
\newcommand{\Phibf}[0]{\bm{\Phi}}
\newcommand{\Psibf}[0]{\bm{\Psi}}
\newcommand{\Pibf}[0]{\bm{\Pi}}
\newcommand{\Omegabf}[0]{\bm{\Omega}}
\newcommand{\omegabf}[0]{\bm{\omega}}
\newcommand{\Thetabf}[0]{\bm{\Theta}}

\newcommand{\Ic}{\mathcal{I}}

\newcommand{\Xe}{\boldsymbol{\mathscr{X}}}

\newcommand{\nnz}{\textnormal{nnz}}

\newcommand{\FFT}{\textnormal{FFT}}

\newcommand{\diag}{\textnormal{diag}}

\renewcommand{\vec}[0]{\operatorname{vec}}

\DeclareMathOperator*{\kr}{\bigodot}
\DeclareMathOperator*{\startimes}{\scalerel*{\circledast}{\sum}}

\newcommand{\rowvec}[1]{\begin{bmatrix} #1 \end{bmatrix}}

\newcommand{\defeq}{\stackrel{\text{\tiny \textnormal{def}}}{=}}  

\newcommand{\giturl}{\url{https://github.com/OsmanMalik/countsketch-matrix-tensor-id}}

\journalname{Advances in Computational Mathematics}

\begin{document}

\title{Fast Randomized Matrix and Tensor Interpolative Decomposition Using CountSketch\thanks{This version of the article has been accepted for publication, after peer review (when applicable) and is subject to Springer Nature’s \href{https://www.springernature.com/gp/open-research/policies/accepted-manuscript-terms}{AM terms of use}, but is not the Version of Record and does not reflect post-acceptance improvements, or any corrections. The Version of Record is available online at: \url{http://dx.doi.org/10.1007/s10444-020-09816-9}}
}

\titlerunning{Fast Randomized Matrix and Tensor ID Using CountSketch}        

\author{Osman Asif Malik \and Stephen Becker}

\institute{O.~A.~Malik \at
	Department of Applied Mathematics, University of Colorado Boulder \\
	\email{osman.malik@colorado.edu} \\
	\and
	S.~Becker \at
	Department of Applied Mathematics, University of Colorado Boulder \\
	\email{stephen.becker@colorado.edu}
}

\date{Received: date / Accepted: date}

\maketitle

\begin{abstract}
	We propose a new fast randomized algorithm for interpolative decomposition of matrices which utilizes \textsc{CountSketch}. We then extend this approach to the tensor interpolative decomposition problem introduced by Biagioni et al.\ (J.~Comput.\ Phys.~281, pp.~116--134, 2015). Theoretical performance guarantees are provided for both the matrix and tensor settings. Numerical experiments on both synthetic and real data demonstrate that our algorithms maintain the accuracy of competing methods, while running in less time, achieving at least an order of magnitude speed-up on large matrices and tensors.
\end{abstract}

\section{Introduction} \label{sec:introduction}

Matrix decomposition is a fundamental tool used to compress and analyze data, and to improve the speed of computations. For data and computational problems involving more than two dimensions, analogous tools in the form of tensors and associated decompositions have been developed \citep{kolda2009}. In many modern applications, matrices and tensors can be very large, which makes decomposing them especially challenging. One approach to dealing with this problems is to incorporate randomization in decomposition algorithms \citep{halko2011}. In this paper, we consider the interpolative decomposition (ID) for matrices, as well as the tensor ID problem. By tensor ID, we mean the tensor rank reduction problem as introduced by \citet{biagioni2015}; we provide an exact definition in Section~\ref{sec:tensor-id}. We make the following contributions in this paper:
\begin{itemize}
	\item We propose a new fast randomized algorithm for matrix ID and provide theoretical performance guarantees.
	\item We propose a new randomized algorithm for tensor ID. To the best of our knowledge, we provide the first performance guarantees for any randomized tensor ID algorithm.
	\item We validate our algorithms on both synthetic and real data.
	\item We propose a small modification to the standard \textsc{CountSketch} formulation which helps avoid certain rank deficiency issues and slightly strengthen our matrix ID results.
\end{itemize}

\subsection{Tensors and the CP Decomposition} \label{sec:tensors-and-CP}

For a more complete introduction to tensors and their decompositions, see the review paper by \citet{kolda2009}.
A \emph{tensor} $\Xe \in \Rb^{I_1 \times I_2 \times \cdots \times I_N}$ is an $N$-dimensional array of real numbers, also called an $N$-way tensor. The number of elements in such a tensor is denoted by $\tilde{I} \defeq \prod_{n=1}^N I_n$. Boldface Euler script letters, e.g.\ $\Xe$, denote tensors of dimension 3 or greater; bold capital letters, e.g.\ $\Xbf$, denote matrices; bold lowercase letters, e.g.\ $\xbf$, denote vectors; and lowercase letters, e.g.\ $x$, denote scalars. Uppercase letters, e.g.\ $I$, are used to denote scalars indicating dimension size. A colon is used to denote all elements along a certain dimension. For example, $\xbf_{m:}$ and $\xbf_{:n}$ are the $m$th row and $n$th column of the matrix $\Xbf$, respectively. If $\jbf$ is a vector of column indices, then $\Xbf_{:\jbf}$ denotes the submatrix of $\Xbf$ consisting of the columns of $\Xbf$ whose indices are listed in $\jbf$. $\Ibf^{(K)}$ denotes the $K \times K$ identity matrix. For a matrix $\Xbf$, $\sigma_i(\Xbf)$ denotes its $i$th singular value, and $\sigma_{\text{max}}(\Xbf)$ and $\sigma_{\text{min}}(\Xbf)$ denote the maximum and minimum singular values, respectively. The condition number of $\Xbf$ is defined as $\kappa(\Xbf) \defeq \sigma_{\text{max}}(\Xbf) / \sigma_{\text{min}}(\Xbf)$. The number of nonzero elements of $\Xbf$ is denoted by $\nnz(\Xbf)$. For positive integers $m$ and $n > m$, let $[m] \defeq \{ 1, 2, \ldots, m \}$ and $[m:n] \defeq \{ m, m+1, \ldots, n \}$. 
The \emph{Hadamard product}, or element-wise product, of matrices is denoted by $\circledast$. The \emph{Khatri--Rao product} of matrices is denoted by $\odot$. 
The \emph{tensor Frobenius norm} is denoted by $\|\Xe\|_\F \defeq \|\vec(\Xe)\|_2$, where $\vec(\Xe)$ flattens the tensor $\Xe$ into a column vector. A norm $\|\cdot\|$ with no subscript will always denote the matrix spectral norm.

The singular value decomposition (SVD) decomposes matrices into a sum of rank-1 matrices \citep{golub2013}. Similarly, the \emph{CP decomposition} decomposes a tensor $\Xe \in \Rb^{I_1 \times I_2 \times \cdots \times I_N}$ into a sum of rank-1 tensors:
\begin{equation} \label{eq:def-X-tensor}
\Xe = \sum_{r=1}^R \lambda_r \abf^{(1)}_{:r} \circ \abf^{(2)}_{:r} \circ \cdots \circ \abf^{(N)}_{:r} = \sum_{r=1}^R \lambda_r \Xe^{(r)},
\end{equation}
where $\circ$ denotes outer product, and each $\Xe^{(r)}$ is a rank-1 tensor. Each $\lambda_r$ is called an \emph{s-value}, each $\Abf^{(n)} = [\abf_{:1}^{(n)} \;\; \abf_{:2}^{(n)} \;\; \cdots \;\; \abf_{:R}^{(n)}]$ is called a \emph{factor matrix}, and all vectors $\abf^{(n)}_{:r}$ have unit 2-norm.
Usually, a tensor $\Xe$ is said to be of rank-$R$ if $R$ is the smallest possible number of terms required in a representation of the form \eqref{eq:def-X-tensor}. We will use the term ``rank'' in a looser sense to mean the (not necessarily minimal) number of rank-1 terms in a representation of the form \eqref{eq:def-X-tensor}.

\subsection{Interpolative Decomposition} \label{sec:id}

\subsubsection{Matrix Interpolative Decomposition} \label{sec:matrix-id}

For a matrix $\Abf \in \Rb^{I \times R}$, a rank-$K$ \emph{interpolative decomposition} (ID) takes the form $\Abf \approx \Abf_{:\jbf} \Pbf$,
where $\Abf_{:\jbf} \in \Rb^{I \times K}$ consists of a subset of $K < R$ columns from $\Abf$, and $\Pbf \in \Rb^{K \times R}$ is a coefficient matrix which is well-conditioned in some sense. The fact that the decomposition is expressed in terms of the columns of $\Abf$ means that $\Abf_{:\jbf}$ inherits properties such as sparsity and non-negativity from $\Abf$. Moreover, expressing the decomposition in terms of columns of $\Abf$ can increase interpretability. Algorithm~\ref{alg:matrix-ID} outlines one method to compute a matrix ID.
\begin{algorithm}[htb!]
	\caption{Matrix ID via QR \citep{voronin2017a}}
	\label{alg:matrix-ID}
	\begin{algorithmic}[1]
		\STATE {\bfseries Input:} $\Abf \in \Rb^{I \times R}$, target rank $K$
		\STATE {\bfseries Output:} $\Pbf \in \Rb^{K \times R}$, $\jbf \in [R]^{K}$ 
		\STATE {Perform rank-$K$ QR factorization $\Abf \Pibf \approx \Qbf^{(1)} \Rbf^{(1)}$}
		\STATE {Define $\jbf \in [R]^{K}$ via $\Ibf_{:\jbf}^{(R)} = \Pibf_{:[K]}$}
		\STATE {Partition $\Rbf^{(1)}$ into two parts: $\Rbf^{(11)} = \Rbf^{(1)}_{:[K]}$, $\Rbf^{(12)} = \Rbf^{(1)}_{:[K+1:R]}$}
		\STATE {Compute $\Pbf^\top = \Pibf \rowvec{\Ibf^{(K)} & (\Rbf^{(11)})^{-1} \Rbf^{(12)}}^\top$}
	\end{algorithmic}
\end{algorithm}
\begin{fact} \label{fact:matrix-ID}
	If the partial QR factorization on line~3 in Algorithm~\ref{alg:matrix-ID} is done using the strongly rank-revealing QR (SRRQR) decomposition developed by \citet{gu1996}, then Algorithm~\ref{alg:matrix-ID} has complexity $O(I R^2)$ \citep{cheng2005}. Moreover, the decomposition it produces satisfies the following properties \citep{martinsson2011}:
\begin{enumerate}[(i)]
	\item Some subset of the columns of $\Pbf$ makes up the $K \times K$ identity matrix, \label{it:P-prop-1}
	\item no entry of $\Pbf$ has an absolute value exceeding 2,
	\item $\|\Pbf\| \leq \sqrt{4K(R-K) + 1}$,
	\item $\sigma_\text{min}(\Pbf) \geq 1$, \label{it:P-prop-4}
	\item $\Abf_{:\jbf}\Pbf = \Abf$ when $K = I$ or $K = R$, and
	\item $\|\Abf_{:\jbf}\Pbf - \Abf\| \leq \sigma_{K+1}(\Abf) \sqrt{4K(R-K) + 1}$ when $K < \min(I,R)$.
\end{enumerate}
\end{fact}
In practice, using a variant of column pivoted QR instead of the SRRQR on line~3 of Algorithm~\ref{alg:matrix-ID} works just as well, and reduces the complexity of the algorithm to $O(KIR)$ \citep{cheng2005}. 

There have been subsequent proposals for randomized versions of matrix ID \citep{liberty2007}. \citet{martinsson2011} propose a variant which incorporates Gaussian random sketching. It computes a sketch $\Ybf = \Omegabf \Abf$, where $\Omegabf \in \Rb^{L \times I}$ ($K < L < I$) is a matrix with iid standard normal entries, and then computes an ID $\Ybf \approx \Ybf_{:\jbf} \Pbf$. The same $\jbf$ and $\Pbf$ then give an ID of $\Abf \approx \Abf_{:\jbf} \Pbf$. 
\citet{woolfe2008} propose a similar fast randomized algorithm which uses a subsampled randomized fast Fourier transform (SRFT) instead of a Gaussian matrix. It computes a sketch $\Ybf = \Sbf_\text{sub} \Fbf \Dbf \Abf$, where $\Dbf \in \Rb^{I \times I}$ is a diagonal matrix with each diagonal entry iid and equal to $+1$ or $-1$ with equal probability, $\Fbf \in \Rb^{I \times I}$ is the fast Fourier transform (FFT), and $\Sbf_\text{sub} \in \Rb^{L \times I}$ is a subsampling operator that randomly samples $L$ rows.

\subsubsection{Tensor Interpolative Decomposition} \label{sec:tensor-id}

\citet{biagioni2015} consider the problem of rank reduction of a CP tensor, which they call tensor ID. Suppose $\Xe \in \Rb^{I_1 \times I_2 \times \cdots \times I_N}$ is an $N$-way tensor with CP decomposition \eqref{eq:def-X-tensor}.
Computing a rank-$K$, $K < R$, tensor ID of $\Xe$ amounts to finding a representation
\begin{equation} \label{eq:X-star}
	\hat{\Xe} = \sum_{k=1}^K \hat{\lambda}_k \Xe^{(j_k)} \approx \Xe,
\end{equation}
where $\jbf \in [R]^K$ contains $K$ unique indices. Tensor ID has many applications. For example, in various algorithms, the rank of discretized separated representations of multivariate functions grows with each iteration, requiring repeated rank reduction of CP tensors \citep{beylkin2002, beylkin2006}. Another example is the algorithm by \citet{reynolds2017} for finding the element of maximum magnitude in a CP tensor which also requires repeated rank reduction.

\citet{biagioni2015} approach the tensor ID problem by considering the matrix
\begin{equation} \label{eq:M-matrix}
	\Mbf 
	= \rowvec{\lambda_1 \vec(\Xe^{(1)}) & \cdots & \lambda_R \vec(\Xe^{(R)})}
	= \bigg(\kr_{n=1}^N \Abf^{(n)}\bigg) \diag(\lambda_1,\ldots,\lambda_R),
\end{equation}
where $\diag(\lambda_1, \ldots, \lambda_R) \in \Rb^{R \times R}$ is a diagonal matrix with entries $\lambda_1,\ldots,\lambda_R$.  
The tensor ID problem can now be reduced to identifying columns of $\Mbf$ using matrix ID. However, when the factor matrices have no special structure, $\Mbf$ has $R \tilde{I}$ elements and is therefore typically infeasible to form. 
One way to tackle this problem is by forming the much smaller Gram matrix $\Mbf^\top \Mbf \in \Rb^{R \times R}$, which can be done using $O(R^2 \sum_{n} I_n)$ flops since
\begin{equation} \label{eq:MTM-formula}
	\Mbf^\top \Mbf = (\Abf^{(1)\top} \Abf^{(1)}) \circledast \cdots \circledast (\Abf^{(N)\top} \Abf^{(N)}),
\end{equation}
compute its symmetric matrix ID, and use it to compute an ID of $\Mbf$. This approach, however, can lead to accuracy issues since the Gram matrix can be ill-conditioned, since $\kappa(\Mbf^\top \Mbf) = \kappa^2(\Mbf)$ \citep{biagioni2015}. 
\citet{biagioni2015} therefore propose a randomized method which avoids the ill-conditioning issue and reduces the complexity. This is done by applying a kind of Gaussian sketch to $\Mbf$, but instead of forming a full Gaussian matrix of size $L \times \tilde{I}$, a matrix of the form
\begin{equation} \label{eq:G-matrix}
	\Omegabf = \bigg(\kr_{n=1}^N \Omegabf^{(n)}\bigg)^\top \in \Rb^{L \times \tilde{I}},
\end{equation}
is used, where each $\Omegabf^{(n)} \in \Rb^{I_n \times L}$ is a matrix with elements that are iid standard normal random variables. The sketch $\Ybf = \Omegabf \Mbf$ can then be computed efficiently without ever forming $\Omegabf$ or $\Mbf$, since $y_{lr} = \lambda_r \prod_{n=1}^N \langle \omegabf^{(n)}_{:l}, \abf^{(n)}_{:r} \rangle$. Note that the elements of $\Omegabf$ in \eqref{eq:G-matrix} are not independent. This means that the theory for Gaussian matrix ID, which requires independence, cannot be used to provide guarantees for sketched matrix ID using $\Omegabf$.

\subsection{Basics of \textsc{CountSketch}} \label{sec:basics-of-CS}

Our proposed method uses a type of sketching called \emph{\textsc{CountSketch}} \citep{charikar2004, clarkson2017}, which we now describe. Let $h : [I] \rightarrow [L]$ be a random map such that each $h(i)$ is iid and $(\forall i \in [I]) (\forall l \in [L])$ $\Pb(h(i) = l) = 1/L$, let $\Phibf \in \Rb^{L \times I}$ be a matrix with $\phi_{h(i)i} = 1$ and all other entries equal to 0, and let $\Dbf \in \Rb^{I \times I}$ be a diagonal matrix with each diagonal entry iid and equal to $+1$ or $-1$ with equal probability.
The \textsc{CountSketch} operator $\Sbf \in \Rb^{L \times I}$ is then defined as $\Sbf = \Phibf \Dbf$.
Applying $\Sbf$ to $\Abf \in \Rb^{I \times R}$ does the following: The matrix $\Dbf$ changes the sign of each row of $\Abf$ with probability $1/2$, and the matrix $\Phibf$ then randomly adds each row of $\Dbf \Abf$ to one of $L$ target rows. Due to the special structure of $\Sbf$, it can be applied implicitly with complexity $O(\nnz(\Abf))$ \citep{clarkson2017}.

Suppose $\Abf$ has the special structure $\Abf = \kr_{n=1}^N \Abf^{(n)} \in \Rb^{\tilde{I} \times R}$, where each $\Abf^{(n)} \in \Rb^{I_n \times R}$. For such matrices, there is a variant of \textsc{CountSketch} which allows computing the sketch of $\Abf$ without ever having to form the full matrix, which can be prohibitively large to store explicitly. This variant is called \textsc{TensorSketch} and is developed by \citet{pagh2013}, \citet{pham2013}, \citet{avron2014} and \citet{diao2018}. It works as follows: 
\begin{itemize}
	\item Define $n$ independent random maps $h_n : [I_n] \rightarrow [L]$ such that each $h(i)$ is iid and $(\forall i \in [I_n]) (\forall l \in [L])$ $\Pb(h_n(i) = l) = 1/L$; and
	\item define $n$ independent random sign functions $s_n : [I_n] \rightarrow \{+1, -1\}$ such that $(\forall i \in [I_n])$ $\Pb(s_n(i) = +1) =\Pb(s_n(i) = -1) = 1/2$.
\end{itemize}
Next, define $H : [I_1] \times [I_2] \times \cdots \times [I_N] \rightarrow [L]$ as
\begin{equation}
H(i_1, i_2, \ldots, i_N) \defeq \Big(\sum_{n=1}^N (h_n(i_n)-1) \mod L\Big) + 1,
\end{equation}
and $S : [I_1] \times [I_2] \times \cdots \times [I_N] \rightarrow \{+1,-1\}$ as
\begin{equation} \label{eq:def-S-function}
S(i_1, i_2, \ldots, i_N) \defeq \prod_{n=1}^N s_n(i_n).
\end{equation}
Notice that each row index of $\Abf$ corresponds to a unique $N$-tuple $(i_1,\ldots,i_N)$. $H$ and $S$ can therefore be considered functions on $[\tilde{I}]$. With this in mind, let $\Dbf_S \in \Rb^{\tilde{I} \times \tilde{I}}$ denote a diagonal matrix with the $i$th diagonal entry equal to $S(i)$. If $H$ and $\Dbf_S$ are used instead of $h$ and $\Dbf$ in the definition of \textsc{CountSketch} above, we get \textsc{TensorSketch}, which we will denote by $\Tbf \in \Rb^{L \times \tilde{I}}$. The reason for choosing this formulation is that it can be computed efficiently using the following formula:
\begin{equation} \label{eq:TA-formula}
\Tbf \Abf = \FFT^{-1} \Big( \startimes_{n=1}^N \FFT(\Sbf^{(n)} \Abf^{(n)}) \Big),
\end{equation}
where each $\Sbf^{(n)} \in \Rb^{L \times I_n}$ is a \textsc{CountSketch} operator defined using $h_n$ and the diagonal matrix $\diag(s_n(1), \ldots, s_n(I_n))$. The formula \eqref{eq:TA-formula} follows from the discussion in Section~A in the supplementary material of \citet{diao2018}. Other good sources for further details on \textsc{TensorSketch} are \citet{pagh2013}, \citet{pham2013} and \citet{avron2014}.

\section{Other Related Work} \label{sec:related-work}

We provided an overview of existing ID algorithms in Section~\ref{sec:id}. The matrix ID is related to the CX and CUR decompositions \citep{drineas2003, drineas2006a, drineas2008, mahoney2009, bien2010, wang2013, boutsidis2017}, also known as skeleton approximations \citep{goreinov1997, goreinov1997a, tyrtyshnikov2000},
and the column subset selection problem \citep{frieze2004, deshpande2006, deshpande2006a, boutsidis2009, deshpande2010, guruswami2012, boutsidis2014}. Like ID, the CX decomposition takes the form $\Abf \approx \Cbf \Xbf$, where $\Cbf$ contains a subset of the columns of $\Abf$. The crucial feature that distinguishes ID from a CX decomposition is the additional conditioning requirements on the coefficient matrix $\Pbf$ in ID; the matrix $\Xbf$ in a CX decomposition is not required to have the properties (i)--(iv) listed in Fact~\ref{fact:matrix-ID} \citep{drineas2008}. A CUR decomposition takes the form $\Abf \approx \Cbf \Ubf \Rbf$, where $\Cbf$ and $\Rbf$ contain a subset of the columns and rows of $\Abf$, respectively. Consequently, setting $\Xbf = \Ubf \Rbf$ would yield a CX decomposition. It is well-known that the matrix $\Xbf$ defined in this manner is typically ill-conditioned \citep{voronin2017a}. Since we require the coefficient matrix $\Pbf$ in our decomposition to be well-conditioned, the available algorithms for CX and CUR decomposition are not useful to us. 

Various randomized algorithms have been utilized in the context of tensor decomposition before. Examples include the works of \citet{wang2015}, \citet{battaglino2018}, and \citet{yang2018} for the CP decomposition; \citet{drineas2007}, \citet{tsourakakis2010}, \citet{dacosta2016} and \citet{malik2018} for the Tucker decomposition; and \citet{zhang2018} and \citet{tarzanagh2018} for t-product based decompositions. Other notable works that use CUR-type algorithms or sampling are e.g.\ those by \citet{mahoney2008}, \citet{caiafa2010}, \citet{oseledets2008} and \citet{friedland2011}. The tensor ID which we consider is different from the various problems solved in these previous papers. The goal of tensor ID is not to compute a tensor decomposition from an arbitrary data tensor. Instead, the purpose of tensor ID is to \emph{compress a tensor which is already in CP format} in an efficient and principled manner. To the best of our knowledge, the only work aside from that by \citet{biagioni2015} which considers randomized tensor ID is the paper by \citet{reynolds2016}. They introduce a randomized alternating least-squares (ALS) algorithm, which is better conditioned but slower than the standard ALS algorithm for CP decomposition. \citet{biagioni2015} conclude that standard ALS is much slower than their Gaussian sketching algorithm. We therefore do not compare our proposed tensor ID to the randomized ALS by \citet{reynolds2016} since it is even slower.

To the best of our knowledge, \textsc{TensorSketch} was the first sketch with theoretical guarantees that could be applied particularly efficiently to matrices like $\Mbf$ in \eqref{eq:M-matrix} with Kronecker structured columns. 
Recently, a number of works have appeared that provide guarantees for other methods designed for efficient sketching of Kronecker structured vectors. 
\citet{sun2018} consider sketches of the form \eqref{eq:G-matrix} where each $\Omegabf^{(n)}$ has sub-Gaussian entries. 
They provide Johnson--Lindenstrauss (JL) style guarantees for the case when the sketch is a Khatri--Rao product of two smaller matrices, i.e., for $N=2$ in \eqref{eq:G-matrix}. 
\citet{rakhshan2020} consider sketches which have tensor train or CP tensor structure, and with core tensors and factor matrices that have Gaussian entries. 
They provide JL style guarantees for these sketches for arbitrary orders of the random tensors. 
Their sketch with CP tensor structure includes the sketch in \eqref{eq:G-matrix} as a special case. 
Another line of work considers the Kronecker fast JL transform, which is a structured variant of the fast JL transform of \citet{ailon2009}. 
It was first proposed by \citet{battaglino2018} with theoretical guarantees later provided by \citet{jin2019} and \citet{malik2020}. 
A variant of this transform is also considered by \citet{iwen2019}.

\section{Fast Randomized Matrix ID Using CountSketch} \label{sec:proposed-matrix-id}

Algorithm~\ref{alg:CS-matrix-ID} explains our proposal for \textsc{CountSketch} matrix ID. Proposition~\ref{prop:CS-matrix-ID} provides guarantees for the method. A proof is provided in Section~\ref{sec:proof-1}, which also contains a more detailed version of the bound in \eqref{eq:prop-1-bound}.

\begin{algorithm}[htb!]
	\caption{\textsc{CountSketch} matrix ID (proposal)}
	\label{alg:CS-matrix-ID}
	\begin{algorithmic}[1]
		\STATE {\bfseries Input:} $\Abf \in \Rb^{I \times R}$, target rank $K$, sketch dimension~$L$
		\STATE {\bfseries Output:} $\Pbf \in \Rb^{K \times R}$, $\jbf \in [R]^{K}$ 
		\STATE {Draw \textsc{CountSketch} matrix $\Sbf \in \Rb^{L \times I}$}
		\STATE {Compute sketch $\Ybf = \Sbf \Abf \in \Rb^{L \times R}$ implicitly}
		\STATE {Compute $[\Pbf, \jbf] = \text{Matrix ID}(\Ybf, K)$ using Algorithm~\ref{alg:matrix-ID}}
	\end{algorithmic}
\end{algorithm}

\begin{proposition}[\textsc{CountSketch} matrix ID] \label{prop:CS-matrix-ID}
	Suppose $I$, $R$ and $K < R$ are defined as in Algorithm~\ref{alg:CS-matrix-ID}. Let $\beta > 1$ be a real number and $L$ a positive integer such that
	\begin{equation} \label{eq:prop-1-cond}
		2 \beta (K^2 + K) \leq L < I.
	\end{equation}
	Suppose that the matrix ID on line 5 of Algorithm~\ref{alg:CS-matrix-ID} utilizes SRRQR. Then, the output $\Pbf$ of Algorithm~\ref{alg:CS-matrix-ID} satisfies properties (\ref{it:P-prop-1})--(\ref{it:P-prop-4}) in Fact~\ref{fact:matrix-ID}. Moreover, the outputs $[\Pbf,\jbf]$ satisfy 
	\begin{equation} \label{eq:prop-1-bound}
		\|\Abf_{:\jbf} \Pbf - \Abf\| \lesssim 2 \sigma_{K+1}(\Abf) \sqrt{KIR}
	\end{equation}
	with probability at least $1 - \frac{1}{\beta}$.
\end{proposition}
The condition in \eqref{eq:prop-1-cond} is very similar to that for SRFT matrix ID by \citet{woolfe2008}. The only difference is that instead of a term of the form $(K^2 + K)$, their work only has a factor $K^2$. In practice, the condition in \eqref{eq:prop-1-cond} is very conservative. We find that a small oversampling factor, e.g.\ $L = K + 10$, works well in practice, producing errors of the same size as the other randomized ID methods.
\begin{remark} \label{remark:adversarial-matrices}
	The semi-coherent matrices defined by \citet{avron2010} are adversarial to CountSketch, and therefore to our proposed method. For such matrices, using $L=K+10$ may result in a large error for our method. Some care is therefore necessary when applying our method together this choice of $L$. In Section~\ref{sec:matrix-experiment-2}, we do extensive testing of our method on real-world matrices to demonstrate that using $L=K+10$ works well in practice. We also provide an example of a matrix with semi-coherent structure on which our method fails when choosing $L$ like this.
\end{remark}
\begin{remark} \label{remark:modified-h}	
	In cases when the target rank $K$ is quite large (e.g.\ $K = R/2$), an issue we encountered is that $\Sbf \Abf$ can be rank deficient due to rank deficiency of $\Sbf$. This issue can be dealt with easily by defining $\Sbf$ slightly differently to ensure that each row contains at least one nonzero element. This is done by the following straightforward modification of the map $h$ in the definition of \textsc{CountSketch} in Section~\ref{sec:basics-of-CS}: Let each $h(i) = v_i$, where $\vbf \in \Rb^{I}$ is a uniform random permutation of the elements of the vector $[1, \cdots, L, x_{L+1}, \cdots, x_{I}]$, where each $x_i \in [L]$ is iid uniformly random. With this modification, the guarantees of Proposition~\ref{prop:CS-matrix-ID} still hold. In fact, the condition in \eqref{eq:prop-1-cond} is slightly improved. We give a precise statement with proof in Section~\ref{sec:proof-2}.
\end{remark}

\section{Extending the Results to Tensor ID} \label{sec:proposed-tensor-id}

Let $\Xe$ and $\hat{\Xe}$ be defined as in \eqref{eq:def-X-tensor} and \eqref{eq:X-star}, respectively. Our approach to the tensor ID problem is similar to that of \citet{biagioni2015}: We sketch the matrix $\Mbf$ in \eqref{eq:M-matrix} without forming it and compute a matrix ID of this sketch. The approximation $\hat{\Xe}$ is then constructed using the rank-1 components of $\Xe$ corresponding to the columns of $\Mbf$ used in the ID of that matrix. The s-values $\hat{\lambda}_1,\ldots,\hat{\lambda}_K$ used in the representation of $\hat{\Xe}$ are then computed as $\hat{\lambda}_k = \lambda_{j_k} \sum_{r=1}^R p_{kr}$, for $k \in [K]$.
The sketch we use is the efficient \textsc{TensorSketch} variant of \textsc{CountSketch}. Algorithm~\ref{alg:CS-tensor-ID} outlines our proposed method for tensor ID. Proposition~\ref{prop:CS-tensor-ID} provides guarantees for the method. A proof is provided in Section~\ref{sec:proof-3}. To the best of our knowledge, there are no previous results like Proposition~\ref{prop:CS-tensor-ID} for randomized tensor ID.
\begin{algorithm}[htb!]
	\caption{\textsc{TensorSketch} tensor ID (proposal)}
	\label{alg:CS-tensor-ID}
	\begin{algorithmic}[1]
		\STATE {\bfseries Input:} CP tensor $\Xe \in \Rb^{I_1 \times I_2 \times \cdots \times I_N}$, target rank $K$, sketch dimension~$L$
		\STATE {\bfseries Output:} rank-$K$ approximation $\hat{\Xe}$
		\STATE {Draw \textsc{TensorSketch} operator $\Tbf \in \Rb^{L \times \tilde{I}}$}
		\STATE {Define $\Mbf$ implicitly as in \eqref{eq:M-matrix}}
		\STATE {Compute sketch $\Ybf = \Tbf \Mbf \in \Rb^{L \times R}$ using \eqref{eq:TA-formula}}
		\STATE {Compute $[\Pbf, \jbf] = \text{Matrix ID}(\Ybf, K)$ using Algorithm~\ref{alg:matrix-ID}}
		\STATE {Compute $\hat{\lambda}_k = \lambda_{j_k} \sum_{r=1}^R p_{kr}$ for $k \in [K]$}
		\STATE {Define CP tensor $\hat{\Xe}$ as in \eqref{eq:X-star}}
	\end{algorithmic}
\end{algorithm}
\begin{proposition}[\textsc{TensorSketch} tensor ID] \label{prop:CS-tensor-ID}
	Suppose $I_1, \ldots, I_N$, $R$ and $K < R$ are defined as in Algorithm~\ref{alg:CS-tensor-ID}. Let $\beta > 1$ be a real number and $L$ a positive integer such that $2 (2 + 3^N) \beta K^2 \leq L < \tilde{I}$.
	Suppose that the matrix ID on line~6 of Algorithm~\ref{alg:CS-tensor-ID} utilizes SRRQR. Then, the output of Algorithm~\ref{alg:CS-tensor-ID} satisfies 
	\begin{equation}
		\|\hat{\Xe} - \Xe\|_\F \lesssim 2 \sigma_{K+1}(\Mbf) R \sqrt{K R \tilde{I}}
	\end{equation}
	with probability at least $1 - \frac{1}{\beta}$.
\end{proposition}
As mentioned in Section~\ref{sec:tensor-id}, an issue with forming and then decomposing $\Mbf^\top \Mbf$ is that it can be ill-conditioned. \citet{biagioni2015} point out that the sketched matrix $\Omegabf \Mbf$ typically is much better conditioned since $\kappa(\Omegabf \Mbf) \leq \kappa(\Omegabf) \kappa(\Mbf)$ and Gaussian matrices are well-conditioned. 
As Proposition~\ref{prop:CS-well-cond} demonstrates, the matrix $\Tbf \Mbf$ is also well-conditioned with high probability, when the sketch dimension $L$ is sufficiently large. 
A proof of Proposition~\ref{prop:CS-well-cond} is provided in Section~\ref{sec:proof-4}.
\begin{proposition} \label{prop:CS-well-cond}
	Let $\beta > 1$ be a real number, and let $L, R$ and $I_1, \ldots, I_N$ be positive integers such that $2 (2 + 3^N) \beta R^2 \leq L$.
	Suppose $\Tbf \in \Rb^{L \times \tilde{I}}$ is a \textsc{TensorSketch} matrix, and $\Mbf \in \Rb^{\tilde{I} \times R}$ is an arbitrary matrix. Then ${\kappa(\Tbf \Mbf) \leq 7 \kappa(\Mbf)}$
	with probability at least $1 - \frac{1}{\beta}$.
\end{proposition}

\section{Complexity Analysis} \label{sec:complexity}

In this section, we compare the complexity of our proposed methods with the other algorithms. We assume all QR factorizations are done using column pivoted QR instead of SRRQR, and ignore the cost of generating random variables. We also assume that $L = K + C$ where $C$ is a small positive integer (e.g.\ $L = K + 10$) since this choice works well in practice. Since $K < R$, and we assume $L = K + C$ for a small constant $C$, we also make the assumption $L < R$. 

The costs of the different steps of Algorithm~\ref{alg:CS-matrix-ID} are as follows:
\begin{itemize}
	\item Computing the sketch $\Ybf = \Sbf \Abf$: $O(\nnz(\Abf))$.
	\item Computing Matrix ID of $\Ybf \in \Rb^{L \times R}$, where $L < R$: $O(L^2 R)$. 
\end{itemize}
The total cost is therefore $O(\nnz(\Abf) + K^2 R)$. The cost of standard matrix ID can be found in Remark~3 of \citet{cheng2005}, and the cost of SRFT matrix ID can be found in Remark~5.4 of \citet{woolfe2008}. The cost of Gaussian matrix ID is straightforward to compute similarly to our computation above. Table~\ref{tab:matrix-ID-complexity} summarize these matrix ID complexities.

For the tensor ID algorithms, we assume the input is an $N$-way rank-$R$ CP tensor of size $I \times \cdots \times I$, and that each factor matrix has the same number of nonzeros, which we denote by $\nnz(\Abf)$.
The costs of the different steps of Algorithm~\ref{alg:CS-tensor-ID} are as follows:
\begin{itemize}
	\item Computing the \textsc{TensorSketched} matrix $\Ybf$: $O(N(\nnz(\Abf) + R L \log L))$.
	\item Computing Matrix ID of $\Ybf \in \Rb^{L \times R}$, where $L < R$: $O(L^2 R)$.
	\item Computing $\hat{\lambda}_1, \ldots, \hat{\lambda}_K$: $O(RK)$. 
\end{itemize}
The total cost is therefore $O(N (\nnz(\Abf) + RK \log K) + K^2 R)$. Although \citet{biagioni2015} do not specify these, the complexities for the Gram matrix approach and Gaussian tensor ID can be computed from the descriptions in their paper. Table~\ref{tab:tensor-ID-complexity} summarize the complexities for the different tensor ID algorithms. The constant $C_\text{mult}$ is the cost of computing one Gram matrix $\Abf^{(n)\top} \Abf^{(n)}$ in \eqref{eq:MTM-formula}, which we assume is the same for each $n$, e.g.\ $C_\text{mult} = I R^2$ if the factor matrices were dense. 

\begin{table}[htb!]
	\caption{Comparison of the complexity for matrix ID algorithms.\label{tab:matrix-ID-complexity}}
	\vskip 0.05 in
	\centering
	\begin{tabular}{ll} 
		\toprule
		Algorithm for matrix ID & Complexity \\
		\midrule
		Standard (Alg.~\ref{alg:matrix-ID}) & $KIR$ \\
		Gaussian & $K \nnz(\Abf) + K^2 R$ \\
		SRFT & $I R \log(K) + K^2 R$ \\
		\textsc{CountSketch} (Proposal, Alg.~\ref{alg:CS-matrix-ID}) & $\nnz(\Abf) + K^2 R$ \\
		\bottomrule
	\end{tabular}
\end{table}

\begin{table}[htb!]
	\caption{Comparison of the complexity for tensor ID algorithms. \label{tab:tensor-ID-complexity}}
	\vskip 0.05 in
	\centering
	\begin{tabular}{ll} 
		\toprule
		Algorithm for tensor ID & Complexity \\
		\midrule
		Gram matrix & $N C_{\text{mult}}+ R^3$ \\
		Gaussian & $N K \nnz(\Abf) + K^2 R$ \\
		\textsc{CountSketch} (Proposal, Alg.~\ref{alg:CS-tensor-ID}) & $N(\nnz(\Abf) + R K \log K) + K^2 R$ \\
		\bottomrule
	\end{tabular}
\end{table}

\section{Numerical Experiments} \label{sec:numerical-experiments}

The numerical experiments are done in Matlab R2018b and C. All results are averages over ten runs in an environment using four cores of an Intel Xeon E5-2680 v3 @2.50GHz CPU and 19 GB of RAM. All code used to generate our results can be found at \giturl, including implementations of our proposed methods. For all randomized methods, we use an oversampling parameter equal to 10 (i.e., $L = K+10$ in Algorithms~\ref{alg:CS-matrix-ID} and \ref{alg:CS-tensor-ID}).

\subsection{Matrix ID Experiments} \label{sec:matrix-id-experiments}

We compare the four methods in Table~\ref{tab:matrix-ID-complexity}. For standard matrix ID, we use the implementation in RSVDPACK\footnote{Available at \\\url{https://github.com/sergeyvoronin/LowRankMatrixDecompositionCodes}.}. For the remaining methods, we use our own Matlab implementations which utilize Matlab's column pivoted QR function. RSVDPACK only supports dense matrices. Moreover, since it is challenging to efficiently construct partial QR decompositions of sparse matrices, we did not attempt to write our own implementation of standard matrix ID for sparse matrices; see Section~11.1.8 of \citet{golub2013} for a discussion about the challenges of sparse QR. We therefore have to convert each sparse input matrix to dense format before applying standard matrix ID from RSVDPACK. Similarly, it is challenging to implement an efficient algorithm for FFT for sparse matrices, or for the accelerated FFT by \citet{woolfe2008}. We therefore also have to convert the input matrix to dense format before applying standard FFT in our implementation of SRFT matrix ID. However, by only sketching a subset of columns of the input matrix at a time, we can avoid having to convert all columns of the matrix to dense format at the same time. In the experiments, we use the modification described in Remark~\ref{remark:modified-h} when implementing our proposed \textsc{CountSketch} matrix ID.
 
Computing the spectral norm of the matrices we consider is not feasible due to their size. Therefore, when computing the error for each matrix decomposition, we utilize the randomized scheme for estimating the spectral norm suggested in Section~3.4 of \citet{woolfe2008}. Letting $E$ be the true error in spectral norm, our estimates $\tilde{E}$ satisfy the following properties: $\tilde{E} \leq E$, and $\Pb(\tilde{E} \geq E / 100) = 1-q$ where $0 < q \ll 1$. In other words, the estimate is smaller than the true spectral norm, but it is unlikely to be much smaller (with ``much smaller'' meaning more than two orders of magnitude smaller). This is good enough for our purposes, since we are primarily interested in comparing the performance of the different methods rather than establishing the exact errors. In the first experiment, $q < \text{2e\textminus2}$, and in the second $q < \text{2e\textminus5}$. 

\subsubsection{Experiment 1: Synthetic Matrices} \label{sec:matrix-experiment-1}

We generate sparse matrices $\Abf \in \Rb^{I \times R}$ with $R = \text{1e+4}$, and density $\nnz(\Abf)/(IR) \approx 0.5\%$. We use different values of $I \in [\text{1e+4}, \text{1e+6}]$. The matrices have a true rank of $2K$, where $K = \text{1e+3}$. Similarly to experiments by \citet{martinsson2011}, we let $\sigma_{i}(\Abf)$, $i \in [K]$, decay exponentially to $10^{-8}$, and then remain constant at $\sigma_i(\Abf) \approx 10^{-8}$, $i \in [K+1:2K]$. 

The results for the first experiment are presented in Figure~\ref{fig:matrix-ID-1}. Standard matrix ID encountered memory issues when $I \geq \text{5e+4}$. For the matrix sizes the standard method could handle, it was more accurate but much slower than the randomized methods. The accuracy of all randomized methods is comparable. Our proposed \textsc{CountSketch} matrix ID is the fastest, achieving a speed-up of about $18 \times$ and $12 \times$ when $I = \text{1e+6}$ compared to Gaussian and SRFT ID, respectively.

\begin{figure}[ht!]
	\centering
	\includegraphics[width=1\columnwidth]{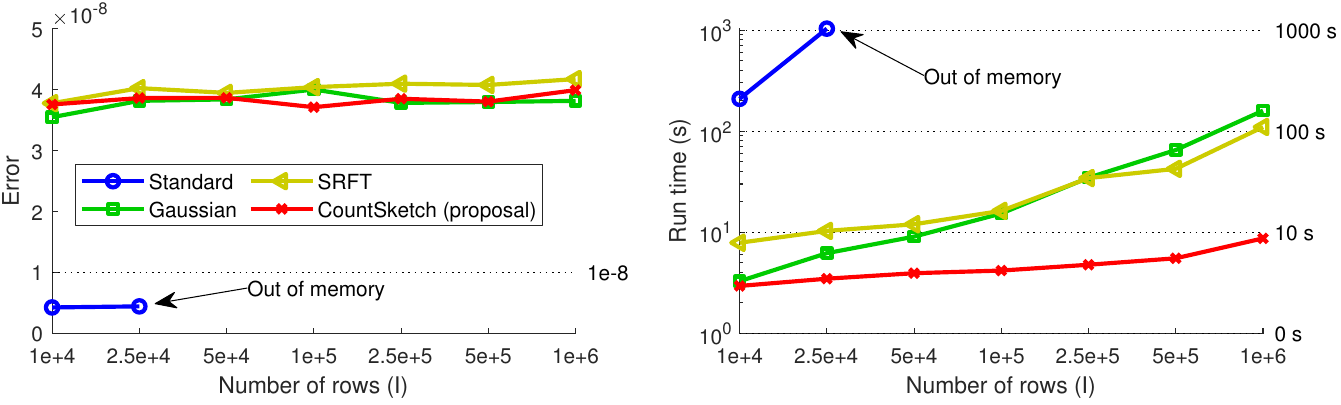}
	\caption{Errors (left) and run times (right) in the synthetic matrix ID experiment. The errors are computed using the randomized spectral norm by \citet{woolfe2008}.}
	\label{fig:matrix-ID-1}
\end{figure}

\subsubsection{Experiment 2: Real-World Matrices} \label{sec:matrix-experiment-2}

We decompose a sparse matrix which comes from a computer vision problem and is part of the SuiteSparse Matrix Collection\footnote{The matrix can be downloaded from \url{https://sparse.tamu.edu/Brogan/specular}.}. The matrix is of size 477,976 by 1,600, contains 7,647,040 nonzero elements, and has a rank of 1,442. We set the target rank to $K = \text{1,442}$. Ideally, the methods should be able to produce decompositions with a very small error. We only attempt this with the three randomized methods, since the matrix is too large for standard matrix ID. Table~\ref{tab:matrix-ID-2-results} shows the result. All methods produce good approximations with a small error. Our proposed \textsc{CountSketch} matrix ID method is much faster than the other algorithms, achieving a speed-up of about $35 \times$ and $31 \times$ compared to Gaussian and SRFT matrix ID, respectively.

\begin{table}[htb!]
	\caption{Errors and run times in the real-world matrix ID experiment. The errors are computed using the randomized spectral norm by \citet{woolfe2008}.\label{tab:matrix-ID-2-results}}
	\centering
	\begin{tabular}{llr} 
		\toprule
		Algorithm for matrix ID & Error & Run time (s) \\
		\midrule
		Gaussian 		& 1.505e\textminus15 & 20.38 \\
		SRFT 			& 1.507e\textminus15 & 18.40 \\
		\textsc{CountSketch} (proposal) 	& 1.504e\textminus15 & 0.59 \\
		\bottomrule
	\end{tabular}
\end{table}

To further support our claim that $L = K + 10$ works well in practice, we have done additional experiments. We consider 20 matrices from the SuiteSparse Matrix Collection\footnote{They are \texttt{landmark}, \texttt{Franz7}, \texttt{ch7-8-b2}, \texttt{ch7-9-b2}, \texttt{ch8-8-b2}, \texttt{mk12-b2}, \texttt{shar\_te2-b1}, \texttt{rel7}, \texttt{relat7b}, \texttt{relat7}, \texttt{abtaha2}, \texttt{abtaha1}, \texttt{specular}, \texttt{photogrammetry2}, \texttt{GL7d12}, \texttt{ch7-6-b2}, \texttt{ch7-7-b2}, \texttt{cis-n4c6-b3}, \texttt{mk11-b2}, \texttt{n4c6-b3}.} , and 3 different target ranks (10\%, 50\% and 90\% of the number of columns). The matrices are of different sizes and come from different application areas. We compare the performance of Gaussian, SRFT and \textsc{CountSketch} (proposed method) matrix ID, all with $L = K + 10$, repeating each experiment 10 times and reporting averages. The results are in Table~\ref{tab:matrix-ID-2-results-2}; on average, our method is \emph{the most accurate} even compared to the Gaussian method which it outperforms in 34 of the 60 tests. Our method outperforms the SRFT method in 57 of the 60 tests. Out of the 26 cases when the Gaussian method outperforms our method, the difference is no more than 7\% in 25 of those cases, and 31\% in one case. Out of the 58 cases when the Gaussian method outperforms the SRFT method, the difference is no more than 14\% in 56 of those cases, and 36\%--45\% in two cases. 
\begin{table}[htb!]
	\caption{Number of experiments out of 60 for which method A is more accurate than method B.\label{tab:matrix-ID-2-results-2}}
	\centering
	\begin{tabular}{lccc}  
		\toprule
		& \multicolumn{3}{c}{Method B} \\
		\cmidrule(r){2-4}
		Method A  		& Our Proposal & SRFT & Gaussian \\
		\midrule
		Our Proposal	& - 	& 57 	& 34 \\
		SRFT       		& 3 	& - 	& 2 \\
		Gaussian   		& 26 	& 58	& - \\
		\bottomrule
	\end{tabular}
\end{table}

As mentioned in Remark~\ref{remark:adversarial-matrices}, semi-coherent matrices are adversarial to \text{CountSketch} and our proposed method. The matrix \verb|soc-sign-bitcoin-otc| in the SuiteSparse Matrix Collection, 
which is the adjacency matrix of a graph, is a concrete example of when choosing $L=K+10$ results in a large error for our method. The semi-coherent structure of this matrix can be revealed by rearranging it so that rows and columns corresponding to nodes in the same strongly connected components of the graph are adjacent in the matrix. Some care is therefore necessary when applying our method together with the rule of thumb $L=K+10$.

\subsection{Tensor ID Experiments}

We compare the three methods in Table~\ref{tab:tensor-ID-complexity}.
We have implemented all methods ourselves in Matlab and C. 

\subsubsection{Experiment 1: Synthetic Tensors} \label{sec:tensor-experiment-1}

We generate sparse $5$-way tensors $\Xe \in \Rb^{I \times \cdots \times I}$ using \eqref{eq:def-X-tensor}, where each factor matrix column $\abf_{:r}^{(n)}$ is a random sparse vector with a density of 1\%, and we use different values of $I \in [\text{1e+3}, \text{1e+5}]$. The number of rank-1 terms is $R = \text{10,000}$, and we use a target rank of 1,000. The values of $\lambda_r$ in \eqref{eq:def-X-tensor} are defined as $\lambda_r \defeq 10^{-\frac{r-1}{R} 8}$ for $r \in [1000]$, and $\lambda_r \defeq 10^{-8}$ for $r \in [1001:R]$. 
The results for the experiment are presented in Figure~\ref{fig:tensor-ID-1}. Gaussian and \textsc{CountSketch} tensor ID achieve similar accuracy. Although the Gram matrix approach has a better accuracy here, it can have issues reaching an error below the square root of machine precision due to poor conditioning; see the example in Section~5.1.1 of \citet{biagioni2015}. Our proposed method is much faster than both other methods for the larger tensors, achieving a speed-up of $46 \times$ over the Gram matrix approach (for $I = \text{2.5e+4}$) and $14 \times$ over Gaussian tensor ID (for $I = \text{1e+5}$).
\begin{figure}[ht!]
	\centering
	\includegraphics[width=1\columnwidth]{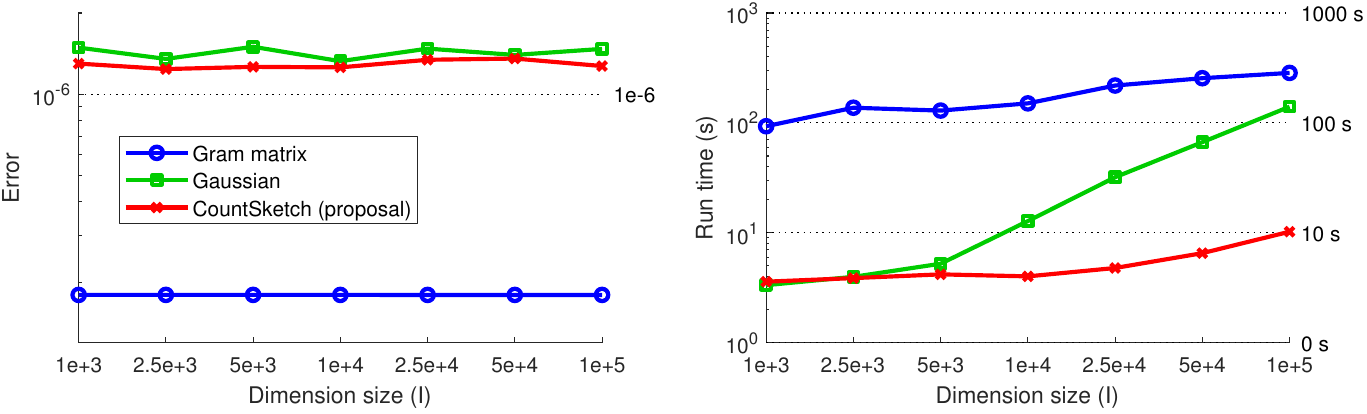}
	\caption{Errors (left) and run times (right) in the synthetic tensor ID experiment. The errors are in Frobenius norm.}
	\label{fig:tensor-ID-1}
\end{figure}

\subsubsection{Experiment 2: Real-World Tensor}

The purpose of this experiment is to show how tensor ID can be useful in a data analysis task. We implement Algorithm~2 by \citet{reynolds2017}, which requires repeated rank reduction, and use it to find the maximum magnitude element in a CP tensor which comes from decomposing streamed data. The rank reduction step is done using tensor ID. The data we consider is a decomposed version of the Enron data set\footnote{The data set is available at \url{http://frostt.io/tensors/enron}.} of size $\text{6,066} \times \text{5,699} \times \text{244,268} \times \text{1,176}$. The Enron data set keeps track of email correspondence between employees at Enron, and the four modes represent sender, receiver, keyword and date. The decomposition has rank 100, and was constructed using the streamed version of SPLATT\footnote{The streamed version of SPLATT is available at \url{https://github.com/ShadenSmith/splatt-stream}.} \citep{smith2018}, 
with the data streamed along the fourth mode (time). As suggested in the documentation of SPLATT-stream, we apply an additional Frobenius norm regularizer with regularization coefficient 1e\textminus2 to the mode-4 factor matrix. We threshold the factor matrices outputted by SPLATT-stream by first normalizing them so that each column have unit 2-norm (the normalization constant is absorbed into the s-values) and then setting all elements with magnitude less than 1e\textminus6 to zero. The relative error introduced by this thresholding is less than 2e\textminus5.

Unlike the previous experiments, the matrices being sketched in this experiment have many rows containing only zeros. We therefore could speed up Gaussian tensor ID by only generating those columns of the Gaussian sketch matrices which are actually multiplied by nonzero elements. We used this improved version of Gaussian tensor ID in the experiment for a more fair comparison. The same modification does not yield a speed-up of Gaussian matrix or tensor ID in the previous experiments since there most rows of the matrices being sketched contain nonzero elements.

Finding the maximum magnitude element using a brute force approach would require computing every nonzero element in the tensor, which would be costly. Using the algorithm by \citet{reynolds2017} together with our \textsc{CountSketch} tensor ID, we find the maximum in 11 seconds. The sketching portion of the algorithm takes $2.6 \times$ more time if Gaussian tensor ID is used instead. We do not compare with the Gram matrix approach since it takes very long to run. With the results in the previous subsection in mind, we believe the speed-up would be more substantial for higher rank tensors. For all ten trials, and both when using \textsc{CountSketch} and Gaussian tensor ID for rank reduction, the same position for the maximum magnitude element is identified each time.

\section{Proofs}

\subsection{Proof of Proposition~\ref{prop:CS-matrix-ID}} \label{sec:proof-1}

Our proof of Proposition~\ref{prop:CS-matrix-ID} is an adaption of the proof for SRFT matrix ID provided by \citet{woolfe2008}. We show that their arguments hold when a \textsc{CountSketch} matrix is used for sketching instead of an SRFT matrix. Although much of our proof is identical to that provided by \citet{woolfe2008}, we choose to include it in detail. The reason for doing this is that the proofs of Propositions~\ref{prop:CS-tensor-ID} and \ref{prop:modified-h} rely on adapting the proof in the present section. Having a detailed proof here therefore makes those subsequent proofs easier to follow.

The following facts will be useful in the proof.
\begin{fact}[Lemma~3.7 in \cite{martinsson2011}] \label{fact:1}
	Let $I$ and $R$ be positive integers with $I \geq R$. Suppose $\Abf \in \Rb^{I \times R}$ is a matrix such that $\Abf^\top \Abf$ is invertible. Then
	\begin{equation}
		\|(\Abf^\top \Abf)^{-1} \Abf^\top\| = \frac{1}{\sigma_R(\Abf)}.
	\end{equation}
\end{fact}

\begin{fact}[Lemma~3.7 in \cite{woolfe2008}] \label{fact:2}
	Let $K$, $L$, $I$ and $R$ be positive integers such that $K \leq R$. Suppose $\Abf \in \Rb^{I \times R}$, $\Bbf \in \Rb^{I \times K}$ is a matrix whose columns constitute a subset of the columns of $\Abf$, $\Pbf \in \Rb^{K \times R}$, $\Xbf \in \Rb^{I \times L}$, and $\Sbf \in \Rb^{L \times I}$. Then
	\begin{equation}
		\|\Bbf \Pbf - \Abf\| \leq \|\Xbf \Sbf \Abf - \Abf\| (\|\Pbf\| + 1) + \|\Xbf\| \|\Sbf \Bbf \Pbf - \Sbf \Abf\|.
	\end{equation}
\end{fact}

\begin{fact}[Lemma~3.9 in \cite{martinsson2011}] \label{fact:3}
	Let $L$, $I$ and $R$ be positive integers. Suppose $\Abf \in \Rb^{I \times R}$, and $\Sbf \in \Rb^{L \times I}$. Then $\sigma_{j}(\Sbf \Abf) \leq \|\Sbf\| \sigma_{j}(\Abf)$ for all $j \in [\min(L, I, R)]$.
\end{fact}

Fact~\ref{fact:atkinson-han} is a special case of a more general statement in \citet{atkinson2009}.
\begin{fact}[Theorem~2.3.1 of \citet{atkinson2009}] \label{fact:atkinson-han}
	Let $\Lbf \in \Rb^{K \times K}$ be a matrix and assume $\|\Lbf\| < 1$. Then $(\Ibf^{(K)} - \Lbf)$ is invertible and
	\begin{equation}
		\|(\Ibf^{(K)} - \Lbf)^{-1}\| \leq \frac{1}{1 - \|\Lbf\|}.
	\end{equation}
\end{fact}

Lemma~\ref{lemma:1} is an adaption of Lemma~4.2 by \citet{woolfe2008}.
\begin{lemma} \label{lemma:1}
	Let $K$, $L$ and $I$ be positive integers such that $K \leq I$. Suppose $\Sbf = \Phibf \Dbf \in \Rb^{L \times I}$ is a \textsc{CountSketch} matrix, and $\Ubf \in \Rb^{I \times K}$ is a matrix with orthonormal columns. Define $\Cbf \in \Rb^{K \times K}$ as
	\begin{equation}
		\Cbf \defeq (\Sbf \Ubf)^\top (\Sbf \Ubf),
	\end{equation}
	and define $\Ebf \in \Rb^{K \times K}$ elementwise as
	\begin{equation} \label{eq:def-E}
		e_{k k'} \defeq \sum_{\substack{i, i' \in [I]\\i \neq i'}} d_{ii} d_{i'i'} u_{ik} u_{i'k'} \Big( \sum_{l \in [L]} \phi_{li} \phi_{li'} \Big).
	\end{equation}
	Then $\Cbf = \Ibf^{(K)} + \Ebf$.
\end{lemma}
\begin{proof}
	For $k, k' \in [K]$,
	\begin{equation} \label{eq:lemma-1-1}
		c_{kk'} = \sum_{l \in [L]} (\Sbf \Ubf)_{lk} (\Sbf \Ubf)_{lk'}.
	\end{equation}
	Since
	\begin{equation}
		(\Sbf \Ubf)_{lk} = \sum_{i \in [I]} \phi_{li} d_{ii} u_{ik},
	\end{equation}
	we can rewrite \eqref{eq:lemma-1-1} as
	\begin{equation}
	\begin{aligned}
		c_{kk'} 
		&= \sum_{l \in [L]} \Big(\sum_{i \in [I]} \phi_{li} d_{ii} u_{ik}\Big) \Big(\sum_{i' \in [I]} \phi_{li'} d_{i'i'} u_{i'k'}\Big) \\
		&= \sum_{i \in [I]} \sum_{l \in [L]} \phi_{li}^2 d_{ii}^2 u_{ik} u_{ik'} + \sum_{\substack{i, i' \in [I]\\i \neq i'}} d_{ii} d_{i'i'} u_{ik} u_{i'k'} \Big( \sum_{l \in [L]} \phi_{li} \phi_{li'} \Big).
	\end{aligned}
	\end{equation}
	The second term on the last line in the equation above is just $e_{kk'}$. Since
	\begin{equation}
		\phi_{l i}^2 = 
		\begin{cases}
			1 & \text{if } h(i) = l, \\
			0 & \text{otherwise},
		\end{cases}
	\end{equation}
	and $d_{ii}^2 = 1$, the first term is just
	\begin{equation}
		\sum_{i \in [I]} \sum_{l \in [L]} \phi_{li}^2 d_{ii}^2 u_{ik} u_{ik'} = \sum_{i \in [I]} u_{ik} u_{ik'} = \langle \ubf_{:k}, \ubf_{:k'}\rangle = 
		\begin{cases}
			1 & \text{if } k = k',\\
			0 & \text{otherwise}.
		\end{cases}
	\end{equation}
	It follows that $\Cbf = \Ibf^{(K)} + \Ebf$.
	\qed
\end{proof}

Lemma~\ref{lemma:2} is an adaption of Lemma~4.3 by \citet{woolfe2008}.
\begin{lemma} \label{lemma:2}
	Let $\alpha$ and $\beta$ be real numbers such that $\alpha, \beta > 1$, and let $K$, $L$ and $I$ be positive integers such that
	\begin{equation} \label{eq:lemma-2-cond}
		\Big(\frac{\alpha}{\alpha-1}\Big)^2 \beta (K^2 + K) \leq L < I.
	\end{equation}
	Suppose $\Sbf = \Phibf \Dbf \in \Rb^{L \times I}$ is a \textsc{CountSketch} matrix, $\Ubf \in \Rb^{I \times K}$ is a matrix with orthonormal columns, and $\Ebf \in \Rb^{K \times K}$ is the matrix defined in \eqref{eq:def-E}. Then
	\begin{equation} \label{eq:E-bound}
		\|\Ebf\| \leq 1 - \frac{1}{\alpha}
	\end{equation}
	with probability at least $1 - \frac{1}{\beta}$.
\end{lemma}
\begin{proof}
	Using the definition in \eqref{eq:def-E}, we have
	\begin{equation} \label{eq:lemma-2-1}
		\Eb[e_{kk'}^2] = \Eb\bigg[ \sum_{\substack{i, i' \in [I]\\i \neq i'}} \sum_{\substack{j, j' \in [I]\\j \neq j'}} d_{ii} d_{i'i'} d_{jj} d_{j'j'} u_{ik} u_{i'k'} u_{jk} u_{j'k'} \Big( \sum_{l \in [L]} \phi_{li} \phi_{li'} \Big) \Big( \sum_{l \in [L]} \phi_{lj} \phi_{lj'} \Big) \bigg].
	\end{equation}
	Note that for each term in the sum above, $i \neq i'$ and $j \neq j'$. This means that unless ($i = j$ and $i' = j'$) or ($i = j'$ and $i' = j$), we have
	\begin{equation}
		\Eb\bigg[d_{ii} d_{i'i'} d_{jj} d_{j'j'} u_{ik} u_{i'k'} u_{jk} u_{j'k'} \Big( \sum_{l \in [L]} \phi_{li} \phi_{li'} \Big) \Big( \sum_{l \in [L]} \phi_{lj} \phi_{lj'} \Big) \bigg] = 0,
	\end{equation}
	since each $d_{ii}$ is independent from all other random variables, and since $\Eb[d_{ii}] = 0$ for all $i \in [I]$. We can therefore rewrite \eqref{eq:lemma-2-1} as 
	\begin{equation} \label{eq:lemma-2-2}
	\begin{aligned}
		\Eb[e_{kk'}^2] 
		&= \sum_{\substack{i, i' \in [I]\\i \neq i'}} \Eb\bigg[d_{ii}^2 d_{i'i'}^2 u_{ik} u_{i'k'} u_{ik} u_{i'k'} \Big( \sum_{l \in [L]} \phi_{li} \phi_{li'} \Big)^2\bigg] \\
		&+ \sum_{\substack{i, i' \in [I]\\i \neq i'}} \Eb\bigg[d_{ii}^2 d_{i'i'}^2 u_{ik} u_{i'k'} u_{i'k} u_{ik'} \Big( \sum_{l \in [L]} \phi_{li} \phi_{li'} \Big)^2\bigg].
	\end{aligned}
	\end{equation}
	The matrix $\Phibf$ has exactly one nonzero entry which is equal to 1 in each column. Consequently,
	\begin{equation}
		\Big( \sum_{l \in [L]} \phi_{li} \phi_{li'} \Big)^2 = 
		\begin{cases}
			1 & \text{if } h(i) = h(i'),\\
			0 & \text{otherwise}.
		\end{cases}
	\end{equation}
	The event $h(i) = h(i')$ happens with probability $\frac{1}{L}$ when $i \neq i'$. If follows that
	\begin{equation} \label{eq:expectation}
		\Eb\bigg[ \Big( \sum_{l \in [L]} \phi_{li} \phi_{li'} \Big)^2 \bigg] = 1 \times \frac{1}{L} + 0 \times \Big(1 - \frac{1}{L}\Big) = \frac{1}{L}.
	\end{equation}
	Using this fact, and the fact that each $d_{ii}^2 = 1$, \eqref{eq:lemma-2-2} simplifies to
	\begin{equation} \label{eq:lemma-2-3}
		\Eb[e_{kk'}^2] = \frac{1}{L} \sum_{\substack{i, i' \in [I]\\i \neq i'}} u_{ik}^2 u_{i'k'}^2 + \frac{1}{L} \sum_{\substack{i, i' \in [I]\\i \neq i'}} u_{ik} u_{i'k'} u_{i'k} u_{ik'}.
	\end{equation}
	Note that
	\begin{equation} \label{eq:lemma-2-4}
		\sum_{\substack{i, i' \in [I]\\i \neq i'}} u_{ik}^2 u_{i'k'}^2 = \sum_{i \in [I]} u_{ik}^2 \sum_{\substack{i' \in [I]\\i' \neq i}} u_{i'k'}^2 \leq \|\ubf_{:k}\|^2 \|\ubf_{:k'}\|^2 = 1.
	\end{equation}
	Moreover,
	\begin{equation}
	\begin{aligned} \label{eq:lemma-2-5}
		&\sum_{\substack{i, i' \in [I]\\i \neq i'}} u_{ik} u_{i'k'} u_{i'k} u_{ik'} 
		= \sum_{i \in [I]} u_{ik} u_{ik'} \sum_{\substack{i' \in [I]\\i' \neq i}} u_{i'k'} u_{i'k} \\
		&\;\;\;\;= \sum_{i \in [I]} u_{ik} u_{ik'} ( \langle \ubf_{:k}, \ubf_{:k'} \rangle - u_{ik}u_{ik'}) \\
		&\;\;\;\;= \langle \ubf_{:k}, \ubf_{:k'} \rangle^2 - \sum_{i \in [I]} u_{ik}^2 u_{ik'}^2 \leq \langle \ubf_{:k}, \ubf_{:k'} \rangle^2 = 
		\begin{cases}
			1 & \text{if } k=k',\\
			0 & \text{otherwise}.
		\end{cases}.
	\end{aligned}
	\end{equation}
	Combining \eqref{eq:lemma-2-3}, \eqref{eq:lemma-2-4} and \eqref{eq:lemma-2-5} yields
	\begin{equation}
		\Eb [e_{kk'}^2] \leq 
		\begin{cases}
			\frac{2}{L} & \text{if } k = k',\\
			\frac{1}{L} & \text{otherwise}.
		\end{cases}
	\end{equation}
	Since
	\begin{equation}
		\|\Ebf\|^2 \leq \|\Ebf\|_\F^2 = \sum_{k, k' \in [K]} e_{kk'}^2,
	\end{equation}
	we have
	\begin{equation}
		\Eb[\|\Ebf\|^2] \leq \sum_{k \in [K]} \Eb[e_{kk}^2] + \sum_{\substack{k, k' \in [K]\\k \neq k'}} \Eb[e_{kk'}^2] \leq \frac{2K}{L} + \frac{K^2-K}{L} = \frac{K^2+K}{L}.
	\end{equation}
	Using Markov's inequality and the condition in \eqref{eq:lemma-2-cond}, we have
	\begin{equation}
		\Pb\Big(\|\Ebf\| \geq 1 - \frac{1}{\alpha}\Big) \leq \Pb\Big( \|\Ebf\|^2 \geq \frac{\beta(K^2 + K)}{L} \Big) \leq \frac{L}{\beta(K^2+K)} \Eb[\|\Ebf\|^2] \leq \frac{1}{\beta}.
	\end{equation}
	Consequently,
	\begin{equation} \label{eq:probability-bound}
		\Pb\Big(\|\Ebf\| \leq 1 - \frac{1}{\alpha}\Big) \geq 1 - \frac{1}{\beta}.
	\end{equation}
	\qed
\end{proof}

Lemma~\ref{lemma:3} is an adaption of Lemma~4.4 by \citet{woolfe2008}.
\begin{lemma} \label{lemma:3}
	Let $\alpha$, $\beta$, $K$, $L$ and $I$ satisfy the same properties as in Lemma~\ref{lemma:2}. Furthermore, suppose $\Sbf$, $\Ubf$, and $\Ebf$ are defined as in Lemma~\ref{lemma:2}, and let $\Cbf \defeq (\Sbf \Ubf)^\top (\Sbf \Ubf)$. If \eqref{eq:E-bound} is true, then the following hold:
	\begin{equation}
		\sigma_1(\Sbf \Ubf) = \sqrt{\|\Cbf\|} \leq \sqrt{2 - \frac{1}{\alpha}},
	\end{equation}
	$\Cbf$ is invertible, and 
	\begin{equation}
		\sigma_K(\Sbf \Ubf) = \frac{1}{\sqrt{\|\Cbf^{-1}\|}} \geq \frac{1}{\sqrt{\alpha}}.
	\end{equation}
\end{lemma}
\begin{proof}
	Using Lemma~\ref{lemma:1} and \eqref{eq:E-bound}, we then have
	\begin{equation}
		\sigma_1(\Sbf \Ubf) = \sqrt{\|\Cbf\|} = \sqrt{\|\Ibf^{(K)} + \Ebf\|} \leq \sqrt{\|\Ibf\| + \|\Ebf\|} \leq \sqrt{1 + 1 - \frac{1}{\alpha}} = \sqrt{2 - \frac{1}{\alpha}}.
	\end{equation}
	Since $\Cbf = \Ibf^{(K)} + \Ebf$ and $\|\Ebf\| < 1$, it follows from Fact~\ref{fact:atkinson-han} that $\Cbf$ is invertible and 
	\begin{equation}
		\|\Cbf^{-1}\| = \|(\Ibf + \Ebf)^{-1}\| \leq \frac{1}{1-\|\Ebf\|} \leq \alpha,
	\end{equation}
	where the last inequality follows from \eqref{eq:E-bound}. Consequently,
	\begin{equation}
		\sigma_K(\Sbf \Ubf) = \frac{1}{\sqrt{\|\Cbf^{-1}\|}} \geq \frac{1}{\sqrt{\alpha}}.
	\end{equation}
	\qed
\end{proof}

Lemma~\ref{lemma:4} is an adaption of Lemma~4.5 by \citet{woolfe2008}.
\begin{lemma} \label{lemma:4}
	Let $L$ and $I$ be positive integers with $L < I$. Suppose $\Sbf \in \Rb^{L \times I}$ is a \textsc{CountSketch} matrix. Then $\|\Sbf\| \leq \sqrt{I}$.	
\end{lemma}
\begin{proof}
	The matrix $\Sbf$ contains $I$ nonzero elements, all of magnitude 1. It follows that $\|\Sbf\|_\F^2 = I$, and hence $\|\Sbf\| \leq \|\Sbf\|_\F \leq \sqrt{I}$.
	\qed
\end{proof}

Lemma~\ref{lemma:4.5} is an adaption of Lemma~4.6 by \citet{woolfe2008}.
\begin{lemma} \label{lemma:4.5}
	Let $\alpha$, $\beta$, $K$, $L$ and $I$ satisfy the same properties as in Lemma~\ref{lemma:2}.
	Suppose $\Sbf \in \Rb^{L \times I}$ is a \textsc{CountSketch} matrix, and $\Abf \in \Rb^{I \times R}$ is an arbitrary matrix. Then, with probability at least $1 - \frac{1}{\beta}$,  there exists a matrix $\Xbf \in \Rb^{I \times L}$ such that
	\begin{equation} \label{eq:lemma-4.5-0}
		\|\Xbf \Sbf \Abf - \Abf\| \leq \sigma_{K+1}(\Abf) \sqrt{\alpha I + 1}
	\end{equation}
	and
	\begin{equation} \label{eq:lemma-4.5-0.5}
		\|\Xbf\| \leq \sqrt{\alpha}.
	\end{equation}
\end{lemma}
\begin{proof}
	Let $\Abf = \Ubf \Sigmabf \Vbf^\top$ be the SVD of $\Abf$, where $\Ubf \in \Rb^{I \times I}$ and $\Vbf \in \Rb^{R \times R}$ are unitary, and $\Sigmabf \in \Rb^{I \times R}$ is diagonal with non-negative entries. Split $\Ubf$ into two matrices $\Ubf^{(1)} \in \Rb^{I \times K}$ and $\Ubf^{(2)} \in \Rb^{I \times (I - K)}$ so that $\Ubf = \rowvec{\Ubf^{(1)} & \Ubf^{(2)}}$. Let $\Zbf^{(1)} = \Sbf \Ubf^{(1)} \in \Rb^{L \times K}$ and $\Zbf^{(2)} = \Sbf \Ubf^{(2)} \in \Rb^{L \times (I-K)}$. Then
	\begin{equation} \label{eq:lemma-4.5-0.6}
		\Sbf \Ubf = \rowvec{\Zbf^{(1)} & \Zbf^{(2)}} \in \Rb^{L \times I}.
	\end{equation}
	Define $\Cbf = \Zbf^{(1)\top} \Zbf^{(1)} \in \Rb^{K \times K}$ and let $\Ebf$ be the corresponding matrix defined in \eqref{eq:def-E}, but in terms of $\Ubf^{(1)}$ instead of $\Ubf$. Then $\Cbf = \Ibf^{(K)} + \Ebf$ according to Lemma~\ref{lemma:1}. For the remainder of the proof, we will assume that $\|\Ebf\| \leq 1 - \frac{1}{\alpha}$, which happens with probability at least $1 - \frac{1}{\beta}$ according to Lemma~\ref{lemma:2}. Then $\Cbf$ is invertible according to Lemma~\ref{lemma:3}. Define $\Gbf^{(-1)} \defeq \Cbf^{-1} \Zbf^{(1)\top} = (\Zbf^{(1)\top} \Zbf^{(1)})^{-1} \Zbf^{(1)\top} \in \Rb^{K \times L}$ and
	\begin{equation} \label{eq:lemma-4.5-1}
		\Xbf \defeq \Ubf 
		\begin{bmatrix}
			\Gbf^{(-1)} \\ \zerobf
		\end{bmatrix} \in \Rb^{I \times L}.
	\end{equation}
	According to Fact~\ref{fact:1} and Lemma~\ref{lemma:3}, it follows that
	\begin{equation} \label{eq:lemma-4.5-2}
		\|\Gbf^{(-1)}\| = \frac{1}{\sigma_K(\Zbf^{(1)})} = \frac{1}{\sigma_K(\Sbf \Ubf^{(1)})} \leq \sqrt{\alpha}.
	\end{equation}
	Combining \eqref{eq:lemma-4.5-1} and \eqref{eq:lemma-4.5-2}, we have
	\begin{equation}
		\|\Xbf\| = \|\Gbf^{(-1)}\| \leq \sqrt{\alpha}.
	\end{equation}
	So \eqref{eq:lemma-4.5-0.5} is satisfied. Next, let $\Thetabf \in \Rb^{K \times K}$ and $\Psibf \in \Rb^{(I - K) \times (I - K)}$ be the matrices in the upper left and lower right corners of $\Sigmabf$, respectively, so that
	\begin{equation} \label{eq:lemma-4.5-3}
		\Sigmabf = 
		\begin{bmatrix}
			\Thetabf & \zerobf \\
			\zerobf & \Psibf
		\end{bmatrix}.
	\end{equation}
	It is easy to verify that
	\begin{equation} \label{eq:lemma-4.5-4}
		\Xbf \Sbf \Abf - \Abf = \Ubf 
		\bigg( 
		\begin{bmatrix} 
			\Gbf^{(-1)} \\ \zerobf 
		\end{bmatrix} 
		\begin{bmatrix} 
			\Zbf^{(1)} & \Zbf^{(2)} 
		\end{bmatrix} 
		 - \Ibf^{(I)}
		\bigg) \Sigmabf \Vbf^\top.
	\end{equation}
	Using \eqref{eq:lemma-4.5-3}, we can further rewrite
	\begin{equation} \label{eq:lemma-4.5-5}
		\bigg( 
		\begin{bmatrix} 
		\Gbf^{(-1)} \\ \zerobf 
		\end{bmatrix} 
		\begin{bmatrix} 
		\Zbf^{(1)} & \Zbf^{(2)} 
		\end{bmatrix} 
		- \Ibf^{(I)}
		\bigg) \Sigmabf
		=
		\begin{bmatrix}
			\zerobf & \Gbf^{(-1)} \Zbf^{(2)} \Psibf \\
			\zerobf & - \Psibf
		\end{bmatrix}.
	\end{equation}
	Note that
	\begin{equation} \label{eq:lemma-4.5-6}
		\left\| 
		\begin{bmatrix}
			\zerobf & \Gbf^{(-1)} \Zbf^{(2)} \Psibf \\
			\zerobf & - \Psibf
		\end{bmatrix} 
		\right\|^2
		\leq \|\Gbf^{(-1)} \Zbf^{(2)} \Psibf\|^2 + \|\Psibf\|^2 \leq \|\Gbf^{(-1)}\|^2 \|\Zbf^{(2)}\|^2 \|\Psibf\|^2 + \|\Psibf\|^2.
	\end{equation}
	From \eqref{eq:lemma-4.5-3}, we know that
	\begin{equation} \label{eq:lemma-4.5-7}
		\|\Psibf\| = \sigma_{K+1}(\Abf).
	\end{equation}
	Moreover, using \eqref{eq:lemma-4.5-0.6}, the fact that $\Ubf$ is unitary, and Lemma~\ref{lemma:4}, we have
	\begin{equation} \label{eq:lemma-4.5-8}
		\|\Zbf^{(2)}\| \leq \|\begin{bmatrix} \Zbf^{(1)} & \Zbf^{(2)} \end{bmatrix}\| = \| \Sbf \Ubf \| = \|\Sbf\| \leq \sqrt{I}.
	\end{equation}
	Combining \eqref{eq:lemma-4.5-4}, \eqref{eq:lemma-4.5-5}, \eqref{eq:lemma-4.5-6}, \eqref{eq:lemma-4.5-7}, \eqref{eq:lemma-4.5-8} and \eqref{eq:lemma-4.5-2} we have 
	\begin{equation}
		\|\Xbf \Sbf \Abf - \Abf\| \leq \sigma_{K+1}(\Abf) \sqrt{\alpha I + 1},
	\end{equation}
	which proves \eqref{eq:lemma-4.5-0}.
	\qed
\end{proof}

We can now prove Proposition~\ref{prop:CS-matrix-ID} in the main manuscript. The proof is an adaption of the discussion in Section~5.1 of \citet{woolfe2008}.
\begin{proof}[Proposition~\ref{prop:CS-matrix-ID}]
	According to Fact~\ref{fact:matrix-ID}, the outputs $\Pbf$ and $\jbf$ computed on line~5 of Algorithm~\ref{alg:CS-matrix-ID} satisfy the following: $\Pbf$ satisfies properties \eqref{it:P-prop-1}--\eqref{it:P-prop-4} in Fact~\ref{fact:matrix-ID}, including 
	\begin{equation} \label{eq:pf-prop-1-2}
		\|\Pbf\| \leq \sqrt{4K(R-K)+1}, 
	\end{equation}
	and
	\begin{equation} \label{eq:pf-prop-1-1}
		\|\Ybf_{:\jbf} \Pbf - \Ybf\| \leq \sigma_{K+1}(\Ybf) \sqrt{4 K (R-K) + 1},
	\end{equation}
	since $K \leq \min(L, R)$. 
	Applying Fact~\ref{fact:2}, we have
	\begin{equation} \label{eq:pf-prop-1-3}
		\|\Abf_{:\jbf} \Pbf - \Abf\| \leq \|\Xbf \Sbf \Abf - \Abf\| (\|\Pbf\| + 1) + \|\Xbf\| \|\Sbf \Abf_{:\jbf} \Pbf - \Sbf \Abf\|,
	\end{equation}
	where $\Xbf \in \Cb^{I \times L}$ is an arbitrary matrix. From Lemma~\ref{lemma:4.5}, with probability at least $1-\frac{1}{\beta}$, we can choose $\Xbf$ such that the bounds in \eqref{eq:lemma-4.5-0} and \eqref{eq:lemma-4.5-0.5} hold. Moreover, since $\Ybf = \Sbf \Abf$, it follows that $\Ybf_{:\jbf} = \Sbf \Abf_{:\jbf}$, and consequently,
	\begin{equation} \label{eq:pf-prop-1-4}
		\|\Sbf \Abf_{:\jbf} \Pbf - \Sbf \Abf\| = \| \Ybf_{:\jbf} \Pbf - \Ybf \|.
	\end{equation}
	Combining \eqref{eq:lemma-4.5-0}, \eqref{eq:lemma-4.5-0.5}, \eqref{eq:pf-prop-1-2}, \eqref{eq:pf-prop-1-1}, \eqref{eq:pf-prop-1-3}, \eqref{eq:pf-prop-1-4}, and Fact~\ref{fact:3} gives that
	\begin{equation} \label{eq:prop-1-full-bound}
		\|\Abf_{:\jbf} \Pbf - \Abf\| \leq \sigma_{K+1}(\Abf) \big( (\sqrt{4K(R-K)+1}+1) \sqrt{\alpha I + 1} + \sqrt{4K(R-K)+1}\sqrt{\alpha I} \big)
	\end{equation}
	with probability at least $1 - \frac{1}{\beta}$. Setting $\alpha = 4$ then yields the same bounds as in the statement in Proposition~\ref{prop:CS-matrix-ID}.
	\qed
\end{proof}

\subsection{Formal Statement and Proof of Claim in Remark~\ref{remark:modified-h}} \label{sec:proof-2}

We express the statement in Remark~\ref{remark:modified-h} in slightly different terms here. Let ${f : [I] \rightarrow [L]}$ be a hybrid deterministic/random function defined as
\begin{equation}
f(i) \defeq 
\begin{cases}
i & \text{if } i \in [L], \\
x_i & \text{if } i \in [L+1:I],
\end{cases}
\end{equation}
where all $x_{i}$ are iid random variables that are uniformly distributed in $[L]$. Furthermore, let $\pi : [I] \rightarrow [I]$ be a uniform random permutation function. We then define $\tilde{h} : [I] \rightarrow [L]$ as $\tilde{h}(i) \defeq f(\pi(i))$. Using $\tilde{h}$ instead of $h$ in the definition of \textsc{CountSketch} ensures that $\Sbf$ is of full rank. The guarantees of Proposition~\ref{prop:CS-matrix-ID} still hold for this modified \textsc{CountSketch}, and in fact the bound in \eqref{eq:prop-1-cond} is slightly improved.

\begin{proposition} \label{prop:modified-h}
	If $\tilde{h}$ defined in this way is used instead of $h$ when defining $\Sbf$ on line~3 in Algorithm~\ref{alg:CS-matrix-ID}, then Proposition~\ref{prop:CS-matrix-ID} still holds, but with the condition in \eqref{eq:prop-1-cond} improved to
	\begin{equation} \label{eq:prop-2-cond}
	2 \Big(1 - \frac{L(L-1)}{I(I-1)}\Big) \beta (K^2 + K) \leq L < I.
	\end{equation}
\end{proposition}

We have not seen anyone else consider this kind of modified \textsc{CountSketch}.

\begin{proof}[Proposition~\ref{prop:modified-h}]
	When using the modified \textsc{CountSketch} matrix proposed in Remark~\ref{remark:modified-h}, the only thing that will change in the proof in Section~\ref{sec:proof-1} is Lemma~\ref{lemma:2}. 
	Notice that going from $h$ to $\tilde{h}$ only impacts $\Phibf$ and not $\Dbf$, and since $\Phibf$ and $\Dbf$ remain independent, the argument that takes us from \eqref{eq:lemma-2-1} to \eqref{eq:lemma-2-2} remains valid for $\tilde{h}$. 
	Indeed, the key conditions that each $d_{ii}$ is independent from all other random variables and $\Eb[d_{ii}] = 0$ remain true when we use $\tilde{h}$ instead of $h$.
	However, the expectation in \eqref{eq:expectation} will change, which impacts \eqref{eq:lemma-2-3}, due to the fact that the probability of the event $\tilde{h}(i) = \tilde{h}(i')$ when $i \neq i'$ is not $\frac{1}{L}$. Note that
	\begin{equation} \label{eq:prob-original}
		\Pb(\tilde{h}(i) = \tilde{h}(i')) = \sum_{l \in [L]} \Pb(\tilde{h}(i) = l, \tilde{h}(i') = l) = \sum_{l \in [L]} \Pb(f(\pi(i)) = l, f(\pi(i')) = l).
	\end{equation}
	We can rewrite
	\begin{equation} \label{eq:prob-split}
	\begin{aligned}
		\Pb(f(\pi(i)) = l, f(\pi(i')) = l) 
		&= \Pb(f(\pi(i)) = l, f(\pi(i')) = l, \pi(i) \in [L], \pi(i') \in [L]) \\
		&+ \Pb(f(\pi(i)) = l, f(\pi(i')) = l, \pi(i) \notin [L], \pi(i') \in [L]) \\
		&+ \Pb(f(\pi(i)) = l, f(\pi(i')) = l, \pi(i) \in [L], \pi(i') \notin [L]) \\
		&+ \Pb(f(\pi(i)) = l, f(\pi(i')) = l, \pi(i) \notin [L], \pi(i') \notin [L]). \\
	\end{aligned}
	\end{equation}
	Notice that the first term on the right hand side of \eqref{eq:prob-split} is zero, since $f$ then will map $\pi(i)$ and $\pi(i')$ to distinct elements. The second and third term in \eqref{eq:prob-split} are equal. Considering the second term, we have
	\begin{equation} \label{eq:prob-term-2} 
	\begin{aligned}
		&\Pb(f(\pi(i)) = l, f(\pi(i')) = l, \pi(i) \notin [L], \pi(i') \in [L]) \\
		&\;\;\;\; = \sum_{j \in [L]} \Pb(f(\pi(i)) = l, f(\pi(i')) = l, \pi(i) \notin [L], \pi(i') \in [L], \pi(i') = j) \\
		&\;\;\;\; = \Pb(f(\pi(i)) = l, \pi(i) \notin [L], \pi(i') = l), \\
	\end{aligned}
	\end{equation}
	since if $\pi(i') \in [L]$, then $f(\pi(i')) = l$ if and only if $\pi(i') = l$. Furthermore,
	\begin{equation} \label{eq:prob-term-2-cont}
	\begin{aligned}
		&\Pb(f(\pi(i)) = l, \pi(i) \notin [L], \pi(i') = l) \\
		&\;\;\;\; = \Pb(f(\pi(i)) = l \mid \pi(i) \notin [L], \pi(i') = l) \; \Pb(\pi(i) \notin [L], \pi(i') = l) \\
		&\;\;\;\; = \Pb(x_{I} = l) \; \Pb(\pi(i) \notin [L], \pi(i') = l) \\
		&\;\;\;\; = \frac{1}{L} \frac{I-L}{I(I-1)},
	\end{aligned}
	\end{equation}
	where the second equality is true since each $x_i$, $i \in [L+1:I]$ is iid. For the fourth term in the right hand side of \eqref{eq:prob-split}, we have
	\begin{equation} \label{eq:prob-term-3} 
	\begin{aligned}
		&\Pb(f(\pi(i)) = l, f(\pi(i')) = l, \pi(i) \notin [L], \pi(i') \notin [L]) \\
		&\;\;\;\; = \Pb(x_{\pi(i)} = l, x_{\pi(i')} = l, \pi(i) \notin [L], \pi(i') \notin [L]) \\
		&\;\;\;\; = \Pb^2(x_{I} = l) \Pb(\pi(i) \notin [L], \pi(i') \notin [L]) \\
		&\;\;\;\; = \frac{1}{L^2} \frac{(I-L) (I-L-1)}{I (I-1)},
	\end{aligned}	
	\end{equation}
	where the second equality again holds since each $x_i$, $i \in [L+1:I]$ is iid and $\pi(i) \neq \pi(i')$. Combining \eqref{eq:prob-original}, \eqref{eq:prob-split}, \eqref{eq:prob-term-2}, \eqref{eq:prob-term-2-cont}, \eqref{eq:prob-term-3}, and using the fact that the second and third term in \eqref{eq:prob-split} are equal, we get
	\begin{equation}
		\Pb(\tilde{h}(i) = \tilde{h}(i')) = \sum_{l \in [L]} 2 \frac{1}{L} \frac{I-L}{I(I-1)} + \frac{1}{L^2} \frac{(I-L) (I-L-1)}{I (I-1)} = \frac{1}{L} - \frac{L-1}{I(I-1)}.
	\end{equation}
	Proceeding with the remainder of the proof of Lemma~\ref{lemma:2} as before, we now get a bound
	\begin{equation}
		\Eb[\|\Ebf\|^2] \leq (K^2 + K) \Big( \frac{1}{L} - \frac{L-1}{I(I-1)} \Big).
	\end{equation}
	Using this new bound and the new condition 
	\begin{equation} \label{eq:new-cond}
		\Big(\frac{\alpha}{\alpha-1}\Big)^2 \Big(1 - \frac{L(L-1)}{I(I-1)}\Big) \beta (K^2 + K) \leq L < I
	\end{equation}	
	together with Markov's inequality, we get
	\begin{equation}
		\Pb\Big(\|\Ebf\| \geq 1 - \frac{1}{\alpha}\Big) \leq \Pb\bigg( \|\Ebf\|^2 \geq \frac{1}{L} \Big( 1 - \frac{L(L-1)}{I(I-1)} \Big) \beta (K^2 + K)\bigg) \leq \frac{1}{\beta},
	\end{equation}
	and consequently 
	\begin{equation}
		\Pb\Big(\|\Ebf\| \leq 1 - \frac{1}{\alpha}\Big) \geq 1 - \frac{1}{\beta}
	\end{equation}
	holds in this case too. 

	All the other lemmas will remain the same, with the only exception that Lemmas~\ref{lemma:3} and \ref{lemma:4.5} now will use the new condition in \eqref{eq:new-cond} instead of the old one in \eqref{eq:lemma-2-cond}. The proof of the proposition itself at the end of Section~\ref{sec:proof-1} will therefore remain identical. When using the modified \textsc{CountSketch}, the statements in Proposition~\ref{prop:CS-matrix-ID} will therefore remain true with the new condition in \eqref{eq:new-cond}. Setting $\alpha = 4$ then yields the desired bound.
	\qed
\end{proof}

\subsection{Proof of Proposition~\ref{prop:CS-tensor-ID}} \label{sec:proof-3}

The following fact will be useful. 
\begin{fact}[Theorem~B.1 in the supplement of \citet{diao2018}] \label{fact:diao}
	Recall that $\tilde{I} = I_1 \cdots I_N$.
	Let $\Tbf \in \Rb^{L \times \tilde{I}}$ be a \textsc{TensorSketch} operator defined as in Section~\ref{sec:basics-of-CS} in terms of $N$ \textsc{CountSketch} operators. 
	\begin{enumerate}[(i)]
		\item Suppose $\Abf$ and $\Bbf$ are matrices with $\tilde{I}$ rows. For $L \geq (2+3^N)/(\varepsilon^2 \delta)$, we have
		\begin{equation}
		\Pb ( \| \Abf^\top \Tbf^\top \Tbf \Bbf - \Abf^\top \Bbf \|_\F^2 \leq \varepsilon^2 \|\Abf\|_\F^2 \|\Bbf\|_\F^2 ) \geq 1-\delta.
		\end{equation}
		\item Suppose $\Mbf \in \Rb^{\tilde{I} \times R}$ is any matrix. If $L \geq R^2(2+3^N)/(\varepsilon^2 \delta)$, then the following holds with probability at least $1-\delta$:
		\begin{equation}
		(\forall \xbf \in \Rb^R) \;\;\;\; (1-\varepsilon) \| \Mbf \xbf\| \leq \|\Tbf \Mbf \xbf\| \leq (1+\varepsilon) \| \Mbf \xbf\|.
		\end{equation}
	\end{enumerate}
\end{fact}

We break the proof into two parts. First, we prove Lemma~\ref{lemma:prop-3-1} which is a variant of Proposition~\ref{prop:CS-matrix-ID} for the case when a \textsc{TensorSketch} operator $\Tbf \in \Rb^{L \times \tilde{I}}$ is used instead of a \textsc{CountSketch} operator. Then we prove the proposition itself.

\begin{lemma} \label{lemma:prop-3-1}
	Let $\alpha$ and $\beta$ be real numbers such that $\alpha, \beta > 1$, and let $K$, $L$, $R$ and $I_1,\ldots,I_N$ be positive integers such that $K \leq R$ and
	\begin{equation} \label{eq:prop-3-cond}
	\Big(\frac{\alpha}{\alpha-1}\Big)^2 (2+3^N) \beta K^2 \leq L < \prod_{n=1}^N I_n.
	\end{equation}
	Suppose that the matrix ID on line 6 of Algorithm~\ref{alg:CS-tensor-ID} utilizes SRRQR. Then the outputs $\Pbf$ and $\jbf$ on that line will satisfy
	\begin{equation} \label{eq:prop-3-mat-bound}
	\begin{aligned}
		\|\Mbf_{:\jbf} \Pbf - \Mbf\| 
		&\leq \sigma_{K+1}(\Mbf) \big( (\sqrt{4K(R-K)+1}+1) \sqrt{\alpha {\textstyle\prod_{n=1}^N I_n} + 1} \\
		&+ \sqrt{4K(R-K)+1}\sqrt{\alpha {\textstyle\prod_{n=1}^N I_n}} \big)
	\end{aligned}
	\end{equation}
	with probability at least $1 - \frac{1}{\beta}$.
\end{lemma}
\begin{proof}
	Recall from Section~\ref{sec:basics-of-CS} that \textsc{TensorSketch} is defined similarly to \textsc{CountSketch}, but using the hash function $H$ instead of $h$, and using the diagonal matrix $\Dbf^{(S)}$ instead of $\Dbf$. Letting $\Phibf^{(H)} \in \Rb^{L \times (I_1 \cdots I_N)}$ be a matrix with $\phi^{(H)}_{H(i)i} = 1$ for $i \in [I_1 \cdots I_N]$, and with all other entries equal to 0, we can write $\Tbf = \Phibf^{(H)} \Dbf^{(S)}$. This means that the proof in Section~\ref{sec:proof-1} largely can be repeated to prove the present lemma. Lemma~\ref{lemma:1} remains true in its present form when \textsc{TensorSketch} is used instead of \textsc{CountSketch}. 
	
	To see that Lemma~\ref{lemma:2} remains true with the new condition when $\Sbf$ is replaced by $\Tbf$, let $\Tbf \in \Rb^{L \times \tilde{I}}$ be a \textsc{TensorSketch} operator, and let $\Ubf \in \Rb^{\tilde{I} \times K}$ be a matrix with orthonormal columns. Define $\Cbf \defeq (\Tbf \Ubf)^\top (\Tbf \Ubf)$, and let $\Ebf$ be defined as in \eqref{eq:def-E}, but in terms of the corresponding quantities from \textsc{TensorSketch}. Then $\Ebf = \Cbf - \Ibf^{(K)}$, according to Lemma~\ref{lemma:1}. Using Fact~\ref{fact:diao}~(i), condition \eqref{eq:prop-3-cond}, and the fact that $\|\Ubf\|_\F^2 = K$, we have
	\begin{equation}
	\begin{aligned}
		&\Pb\Big(\|\Ebf\| \leq 1 - \frac{1}{\alpha}\Big) \geq  \Pb\Big(\|\Ebf\|_\F \leq 1 - \frac{1}{\alpha}\Big) \\
		&= \Pb\bigg(\|(\Tbf \Ubf)^\top (\Tbf \Ubf) - \Ibf^{(K)}\|_\F^2 
		\leq \Big(1 - \frac{1}{\alpha}\Big)^2\bigg) \geq 1 - \frac{1}{\beta}.
	\end{aligned}
	\end{equation}
	
	All the other lemmas will remain the same when $\Tbf$ is used instead of $\Sbf$, with the only exception that Lemmas~\ref{lemma:3} and \ref{lemma:4.5} now will use the new condition in \eqref{eq:prop-3-cond} instead of the old one in \eqref{eq:lemma-2-cond}. Using exactly the same arguments as in the proof of Proposition~\ref{prop:CS-matrix-ID} at the end of Section~\ref{sec:proof-1} will therefore give the bound in \eqref{eq:prop-3-mat-bound}.
	\qed
\end{proof}

\begin{proof}[Proposition~\ref{prop:CS-tensor-ID}]
	Recall that $\Xe$ and $\hat{\Xe}$ are defined as in \eqref{eq:def-X-tensor} and \eqref{eq:X-star}, respectively, and the coefficients $\hat{\lambda}_1, \ldots, \hat{\lambda}_K$ are defined as  
	\begin{equation}
		\hat{\lambda}_k = \lambda_{j_k} \sum_{r=1}^R p_{kr}
	\end{equation}
	for $k \in [K]$. We then have
	\begin{equation} \label{eq:tensor-error}
	\begin{aligned}
		&\|\hat{\Xe} - \Xe\|_\F 
		= \Big\|\sum_{k\in[K]} \hat{\lambda}_k \Xe^{(j_k)} - \sum_{r\in[R]} \lambda_r \Xe^{(r)}\Big\|_\F \\
		&\;\;\;\; = \Big\| \sum_{k\in[K]} \Big(\lambda_{j_k} \sum_{r\in[R]} p_{kr}\Big) \Xe^{(j_k)} - \sum_{r\in[R]} \lambda_r \Xe^{(r)} \Big\|_\F \\
		&\;\;\;\; = \Big\| \sum_{r\in[R]} \Big( \sum_{k\in[K]} \lambda_{j_k} \Xe^{(j_k)} p_{kr} - \lambda_r \Xe^{(r)}\Big) \Big\|_\F.
	\end{aligned}
	\end{equation}
	Letting $\Ic \defeq [I_1] \times \cdots \times [I_N]$, we have
	\begin{equation} \label{eq:tensor-error-inequality}
	\begin{aligned}
		&\Big\| \sum_{r\in[R]} \Big( \sum_{k\in[K]} \lambda_{j_k} \Xe^{(j_k)} p_{kr} - \lambda_r \Xe^{(r)}\Big) \Big\|_\F^2 
		= \sum_{\ibf \in \Ic} \bigg( \sum_{r\in[R]} \Big( \sum_{k\in[K]} \lambda_{j_k} x^{(j_k)}_\ibf p_{kr} - \lambda_r x^{(r)}_\ibf \Big) \bigg)^2 \\
		&\;\;\;\; \leq \sum_{\ibf \in \Ic} R \sum_{r\in[R]} \Big( \sum_{k\in[K]} \lambda_{j_k} x^{(j_k)}_\ibf p_{kr} - \lambda_r x^{(r)}_\ibf \Big)^2 \\
		&\;\;\;\; = R \|\Mbf_{:\jbf} \Pbf - \Mbf\|_\F^2,
	\end{aligned}
	\end{equation}
	where the inequality follows from Cauchy--Schwarz inequality. Combining \eqref{eq:tensor-error} and \eqref{eq:tensor-error-inequality} we get
	\begin{equation} \label{eq:tensor-error-bound}
		\|\hat{\Xe} - \Xe\|_\F \leq \sqrt{R} \|\Mbf_{:\jbf} \Pbf - \Mbf\|_\F \leq R \|\Mbf_{:\jbf} \Pbf - \Mbf\|,
	\end{equation}	
	where the second inequality is a well-known relation (see e.g.\ equation (2.3.7) in \cite{golub2013}). Combining \eqref{eq:tensor-error-bound} and Lemma~\ref{lemma:prop-3-1} gives that
	\begin{equation}
	\begin{aligned}
		\|\hat{\Xe} - \Xe\|_\F 
		&\leq \sigma_{K+1}(\Mbf) R \big( (\sqrt{4K(R-K)+1}+1) \sqrt{\alpha {\textstyle\prod_{n=1}^N I_n} + 1} \\
		&+ \sqrt{4K(R-K)+1}\sqrt{\alpha {\textstyle\prod_{n=1}^N I_n}} \big)
	\end{aligned}
	\end{equation}
	with probability at least $1 - \frac{1}{\beta}$. Setting $\alpha = 4$ then yields the same bounds as in the statement in Proposition~\ref{prop:CS-tensor-ID}.
	\qed
\end{proof}

\subsection{Proof of Proposition~\ref{prop:CS-well-cond}} \label{sec:proof-4}

\begin{proof}
	Note that $\Tbf \Mbf$ is of size $L \times R$, with $L > R$. So $\sigma_R(\Tbf \Mbf)$ is the smallest singular value of $\Tbf \Mbf$. Suppose
	\begin{equation}
	\Big(\frac{\alpha}{\alpha-1}\Big)^2 \beta R^2 (2 + 3^N) \leq L.
	\end{equation}	
	To simplify notation, let $\varepsilon \defeq 1-1/\alpha$.
	Using Theorem~8.6.1 in \cite{golub2013} and Fact~\ref{fact:diao}~(ii), we have that with probability at least $1 - \frac{1}{\beta}$, the following hold:
	\begin{equation}
		\sigma_1(\Tbf\Mbf) = \max_{\|\xbf\|=1} \|\Tbf \Mbf \xbf\| \leq (1+\varepsilon) \max_{\|\xbf\|=1} \|\Mbf \xbf\| = (1+\varepsilon) \sigma_\text{max}(\Mbf),
	\end{equation}
	and
	\begin{equation}
		\sigma_R(\Tbf \Mbf) = \min_{\|\xbf\|=1} \|\Tbf \Mbf \xbf\| \geq (1-\varepsilon) \min_{\|\xbf\|=1} \|\Mbf \xbf\| = (1-\varepsilon) \sigma_\text{min}(\Mbf).
	\end{equation}
	We therefore have
	\begin{equation}
		\kappa(\Tbf \Mbf) 
		= \frac{\sigma_1(\Tbf \Mbf)}{\sigma_R(\Tbf \Mbf)} 
		\leq \frac{(1+\varepsilon) \sigma_\text{max}(\Mbf)}{(1-\varepsilon) \sigma_\text{min}(\Mbf)}
		= \frac{(2-\frac{1}{\alpha}) \sigma_\text{max}(\Mbf)}{\frac{1}{\alpha} \sigma_\text{min}(\Mbf)} 
		= (2\alpha - 1) \kappa(\Mbf),
	\end{equation}
	with probability at least $1 - \frac{1}{\beta}$. Setting $\alpha = 4$ gives us the bounds in Proposition~\ref{prop:CS-well-cond}.
	\qed
\end{proof}

\section{Conclusion} \label{sec:conclusion}

We have presented a new fast randomized algorithm for computing matrix ID, which utilizes \textsc{CountSketch}. We have then shown how this method can be extended to computing the tensor ID of CP tensors. For both the matrix and tensor settings, we provided performance guarantees. To the best of our knowledge, we provide the first performance guarantees for any randomized tensor ID algorithm. We conducted several numerical experiments on both synthetic and real data. These experiments showed that our algorithms maintain the same accuracy as other randomized methods, but with a much shorter run time, running at least an order of magnitude faster on the larger matrices and tensors.

\bibliography{library-zotero}
\bibliographystyle{plainnat}

\end{document}